\documentclass[a4paper]{amsart}

\usepackage{amsmath,amssymb}

\allowdisplaybreaks
\newtheorem{theorem}{Theorem}[section]
\newtheorem{proposition}[theorem]{Proposition}
\newtheorem{corollary}[theorem]{Corollary}
\newtheorem{lemma}[theorem]{Lemma}

\newtheorem{preremark}[theorem]{Remark}
\newtheorem{predefinition}[theorem]{Definition}
\newtheorem{preexample}[theorem]{Example}
\newtheorem{prenotation}[theorem]{Notation}
\newtheorem{preconjecture}[theorem]{Conjecture}

\newenvironment{remark}{\begin{preremark}\rm}{\end{preremark}}
\newenvironment{definition}{\begin{predefinition}\rm}
	{\end{predefinition}}
\newenvironment{example}{\begin{preexample}\rm}{\end{preexample}}

\newcommand{\m}{{\mathfrak{m}}}

\newcommand{\M}{{\mathfrak{M}}}
\newcommand{\N}{{\mathfrak{N}}}
\newcommand{\q}{{\mathfrak{q}}}
\newcommand{\Q}{{\mathfrak{Q}}}

\newcommand{\X}{{\mathbb{X}}}

\newcommand{\QQ}{{\mathbb{Q}}}
\newcommand{\ZZ}{{\mathbb{Z}}}
\newcommand{\NN}{{\mathbb{N}}}
\renewcommand{\AA}{{\mathbb{A}}}

\newcommand{\cF}{{\mathcal{F}}}
\newcommand{\cG}{{\mathcal{G}}}
\newcommand{\cM}{{\mathcal{M}}}
\newcommand{\cN}{{\mathcal{N}}}

\newcommand{\cK}{{\mathcal{K}}}
\newcommand{\cD}{{\mathcal{D}}}
\newcommand{\cE}{{\mathcal{E}}}

\let\epsilon=\varepsilon

\def\phi{{\varphi}}
\let\Psi=\varPsi
\let\Phi=\varPhi
\let\theta=\vartheta
\let\rho=\varrho

\def\LT{\mathop{\rm LT}\nolimits}

\newcommand{\HdR}{H_{\rm dR}}
\newcommand{\charac}{\mathop{\rm char}\limits}
\newcommand{\id}{\mathop{\rm id}\nolimits}
\newcommand{\Ker}{\mathop{\rm Ker}\nolimits}
\renewcommand{\Im}{\mathop{\rm Im}\nolimits}
\newcommand{\Coker}{\mathop{\rm Coker}\nolimits}
\newcommand{\Spec}{\mathop{\rm Spec}\nolimits}
\newcommand{\Rad}{\mathop{\rm Rad}\nolimits}
\newcommand{\DF}{\mathop{\rm DF}\nolimits}

\newcommand{\gr}{{\mathop{\rm gr}}}
\newcommand{\ini}{{\mathop{\rm in}}}
\newcommand{\ord}{{\mathop{\rm ord}}}

\newcommand{\Supp}{{\mathop{\rm Supp}\nolimits}}

\newcommand{\soc}{{\mathop{\rm Soc}}}

\let\To=\longrightarrow
\def\TTo#1{\mathop{\longrightarrow}\limits ^{#1}}

\def\tfrac #1#2{{\textstyle\frac{#1}{#2}}}
\def\tsum_#1^#2{{\textstyle\sum\limits_{#1}^{#2}}}

\def\tprod_#1^#2{{\textstyle\prod\limits_{#1}^{#2}}}

\renewcommand{\dot}{{\scriptscriptstyle\bullet}}

\begin{document}
\title{On the De Rham Cohomology of Zero-Dimensional Schemes}

\author{Martin Kreuzer}
\address{Fakult\"at f\"ur Informatik und Mathematik, Universit\"at
	Passau, D-94030 Passau, Germany}
\email{Martin.Kreuzer@uni-passau.de}

\author{Le Ngoc Long}
\address{Department of Mathematics,
	University of Education - Hue University, 34 Le Loi Street, Hue City, Vietnam}
\email{lelong@hueuni.edu.vn}

\date{\today}

\begin{abstract}
The (naive) de Rham cohomology of a zero-dimensional scheme is the homology
of the K\"ahler differential algebra of its coordinate ring, viewed as a
complex. Since it is well-known to vanish in higher degrees if the ring is quasi-homogeneous,
we concentrate on the affine case. For a general affine $K$-algebra $R=P/I$,
where $\charac(K)=0$ and $P=K[x_1,\dots,x_n]$, we prove a vanishing theorem for $H_{\rm dR}^m(R)$ based
on the shape of a Macaulay basis of $dI\wedge \Omega^m_{P/K}$.

For an Artinian local algebra $A=P/\langle f_1,\dots,f_n\rangle$, where $\{f_1,\dots,f_n\}$ is a
super regular sequence, we show that $H_{\rm dR}^\dot(A)$ is non-trivial in general, but
trivial when the natural system of generators of the relation module of $\Omega^m_{A/K}$
is a standard basis. Moreover, we provide a detailed study of the dimension of $H_{\rm dR}^0(A)=\Ker(d_A)$
for Artinian local rings $A=P/I$.

The case of arbitrary affine zero-dimensional schemes is reduced to this case using a Galois splitting,
and as a result we obtain the de Rham cohomology of a fat point scheme. Many explicitly computed
examples and counterexamples support the results and indicate how subtle the de Rham cohomology
of a zero-dimensional scheme is in general.
\end{abstract}

\keywords{De Rham cohomology, zero-dimensional scheme, differential module, 
differential algebra, filtration}

\subjclass[2010]{Primary 13N05; Secondary  13D02, 13E10, 14F40}

\maketitle

%
%
\section{Introduction}

The use of differential algebra techniques in the study of 0-dimensional
subschemes of affine and projective spaces was introduced by G.\ de Dominicis and the first author
in~\cite{DK1999}, chiefly by looking at the K\"ahler differential module of a reduced 
0-dimensional scheme, i.e., a finite set of points, in a projective space. 
Later the current authors, together with T.N.K.\ Linh,
extended these methods to the K\"ahler differential algebras of arbitrary 0-dimensional schemes
(cf.~\cite{KLL2019,KLL2021,KLL2025}). Here we continue and extend our investigations
by considering the de Rham cohomology of 0-dimensional schemes. 

To make sense of this idea, we need to clarify it. First of all, we are not looking at the
\textit{algebraic de Rham cohomology} in the sense of Hartshorne (cf.~\cite{Har1975}) 
which is tailored to capture topological aspects of algebraic varieties and is 
thus not interesting for 0-dimensional schemes. Rather, we are looking
at the \textit{naive de Rham cohomology} which is defined as the homology of the complex
given by the K\"ahler differential algebra.
Secondly, it is known that the naive de Rham cohomology vanishes for standard graded $K$-algebras
(and, more generally, for quasi-homogeneous rings) in higher degrees and is equal to the base 
field in degree zero. Since the homogeneous coordinate ring of a 0-dimensional subscheme of a projective space
is standard graded, it becomes clear that we have to embed the 0-dimensional scheme into an affine space
and study the naive de Rham homology (which we simply called the de Rham cohomology from now on)
of the affine coordinate ring. Moreover, as the affine coordinate ring of a 0-dimensional scheme
is the direct product of its local rings, we are led to studying the de Rham cohomology of 
local Artinian algebras. At this point the investigation becomes truly subtle and intriguing.

Let us describe our findings by going through them section by section. In Section~\ref{sec2}
we introduce the setting and recall the basic definitions and results. We work over a perfect
field~$K$ which is frequently assumed to have characteristic zero. Given an affine $K$-algebra
$R=K[x_1,\dots,x_n]/I$, the exterior algebra $\Omega^\dot_{R/K} = \Lambda^\dot\, \Omega^1_{R/K}$
over its K\"ahler differential module $\Omega^1_{R/K}$ is called its K\"ahler differential algebra.
Using exterior differentiation, it can be viewed as a complex, and the (naive) de Rham
cohomology $\HdR^\dot(R/K)$ of $R/K$ is the homology of this complex. To simplify the notation, let us drop
$/K$ when the base field is clear.

A first consequence of Poincar\'e's Lemma, which says that $\HdR^\dot(K[x_1,\dots,x_n])$ is trivial
in characteristic zero, is that the highest de Rham cohomology satisfies $\HdR^n(R)=0$.
After recalling some functorial properties of $\HdR^\dot(R)$, we introduce the Euler-Koszul complex
of~$R$ (cf.~Def.~\ref{def:Euler}) and explain how it can be used to prove that the de Rham cohomology
of positively graded rings is trivial (see Prop.~\ref{prop:dRgraded}). In the final part of this section
we present a new criterion which can be used to show that certain de Rham cohomology groups vanish.
Namely, let $\charac(K)=0$, let $P=K[x_1,\dots,x_n]$, and assume that the submodule
$dI\wedge \Omega^m_P$ of~$\Omega^{m+1}_P$ has a Macaulay basis of the form $\{df_1\wedge w_1,\dots,
df_r\wedge w_r\}$ with $f_i\in I$ and $w_i \in \Omega^m_P$. (A Macaulay basis is a system of generators
whose degree forms generate the degree form module, see~\cite{KR2005}, Sect.~4.2.B.) Then $\HdR^m(R)=0$.

In view of the introductory discussion above, it is apparent that Section~\ref{sec3},
in which we discuss the de Rham cohomology of Artinian local $K$-algebras $A=K[x_1,\dots,x_n]/I$,
is central to this paper. If $I=\M^\nu$ is a power of the maximal ideal $\M = \langle x_1,\dots,x_n\rangle$,
the scheme $\Spec(A)$ is commonly called a \textit{fat point}. Since the ideal $\M^\nu$ is homogeneous,
the de Rham cohomology of a fat point, and in particular of a reduced point, 
is trivial (see Prop.~\ref{prop:dRFatPoint}).

As soon as we move even slightly away from this case, the situation becomes extraordinarily subtle
and tricky. In Subsection~3.2 we assume $\charac(K)=0$ and look at Artinian local complete intersections
$A=K[x_1,\dots,x_n]/\langle f_1,\dots,f_n\rangle$. Yet, for such nice rings, the kernel of the universal
derivation $d_A:\ A  \longrightarrow \Omega^1_A$ can be larger than~$K$, i.e., we can have
$\dim_K(\HdR^0(A))>1$, and the Euler characteristic of~$\Omega^\dot_A$ can be larger than~1.
Even if we assume that $\{f_1,\dots,f_n\}$ is a super-regular sequence, i.e., that the
initial forms $\{\ini(f_1),\dots,\ini(f_n)\}$ are a homogeneous regular sequence, the problems persist
(see Examples~\ref{ex:KerD3dim},~\ref{ex:EulerChar2}, and~\ref{ex:badSSCI}).

The deeper reason behind the non-vanishing of the de Rham cohomology in these cases is
the following phenomenon. Let us equip~$A$ and~$\Omega^\dot_A$ with the $\m$-adic filtrations,
where $\m=\langle \bar{x}_1,\dots,\bar{x}_n\rangle$ and the induced filtration on~$\Omega^\dot_A$
is denoted by~$\cG$. Then there is a canonical surjective $\gr_\m(A)$-DG-algebra homomorphism
$$
\Phi:\; \Omega^\dot_{\gr_\m(A)} \;\longrightarrow\; \gr_\cG (\Omega^\dot_A)
$$
(see Lemma~\ref{lem:Phi}). The failure of this map to be bijective introduces non-trivial differential 
constants, i.e., non-constant elements in $\Ker(d_A)$, or non-vanishing higher de Rham cohomology modules.
One condition which forces the map~$\Phi$ to be an isomorphism for a 0-dimensional super-strict
complete intersection $A=K[x_1,\dots,x_n]/\langle f_1,\dots,f_n\rangle$ is to require that
the basis elements of the relation module $\sum_{i=1}^n f_i\, \Omega^m_P + df_i \wedge \Omega^{m-1}_P$
of~$\Omega^m_A$ form a standard basis (cf.\ Prop.~\ref{prop:IsomGrOmega}). Then the de Rham cohomology
of such super-strict complete intersections is indeed trivial (see Theorem~\ref{thm:localSSCI}).

In Subsection~3.3 we examine the space $\HdR^0(A)=\Ker(d_A)$ of differential constants of
an Artinian local algebra in detail.  We note that for a super-strict complete intersection
with non-trivial differential constants, the natural conjecture that the socle should be a differential constant,
turns out to be wrong (see Example~\ref{ex:SocEscapeKer}). Then our main result in this subsection,
Theorem~\ref{thm:ZerothDeRham}, provides several formulas and estimates for the dimension of $\HdR^0(A)$.
The most useful ones are a bound by the dimension of the kernel of $\gr(d_A)$
and by $1+\dim_K H^1(\cK^\dot)$, where~$\cK^\dot$ is the kernel of~$\Phi$.
As a consequence, we obtain a hierarchy of conditions which force the differential constants to
consists of the elements of~$K$ only (see Cor.~\ref{cor:H0conditions}).

Based on these local insights, we move to the de Rham cohomology of the coordinate ring
of a 0-dimensional scheme in Section~\ref{sec4}. As mentioned above, for a projective 0-dimensional
scheme both the de Rham cohomology of the homogeneous coordinate ring and its Artinian reduction
are trivial (see Cor.~\ref{cor:DeRhamCohProj} and~\ref{cor:DRofArtRed}). 

For the affine case, we can immediately reduce everything to a scheme whose support
is one point. However, that point need not be $K$-rational. To reduce the investigation even further
to $K$-rational points, we extend the base field and use the Galois splitting (see Prop.~\ref{prop:ReduceToKRat}).
Then we can move the individual $K$-rational points to the origin via linear changes
of coordinates and reach the setting of Section~\ref{sec3}. As a result, we obtain the de Rham
cohomology of fat points schemes, i.e., schemes whose vanishing ideal is of the form
$\M_1^{\nu_1} \cap \cdots \cap \M_s^{\nu_s}$ with maximal ideals $\M_i$ of~$P$ and $\nu_i\ge 1$
(see Prop.~\ref{prop:HdRofFatPoints}). This includes that case of reduced 0-dimensional 
schemes (see Cor.~\ref{cor:dRofPointSets}). Finally, we review and interpret an example of
A.G.\ Alexandrov (cf.~\cite{Ale}) and show by example that 
the de Rham cohomology of the affine coordinate ring may be trivial even if the local rings
of the scheme are not quasi-homogeneous (see Examples~\ref{ex:alex} and~\ref{ex:NonQuasiHomog}).

For the general notation and terminology used here, we refer to the 
books~\cite{KR2000,KR2005,KR2016}. Many of the insights in this paper are supported
by explicitly computed examples.
These examples were calculated using the freely available 
computer algebra system ApCoCoA (see~\cite{ApCoCoA}), but can be readily verified using many other
systems. The source code and output of the examples is available from the authors upon request.

\bigbreak
%
%

\section{The de Rham Cohomology of Finitely Generated Algebras}
\label{sec2}

In the following we let~$K$ be a perfect field and~$R$ 
a finitely generated $K$-algebra. Then $R$ can be 
represented as $R\cong P/I$, where $P=K[x_1, \dots, x_n]$ 
is a polynomial ring over~$K$ and~$I$ is an ideal in~$P$. 

\begin{definition}
Let $\mu:\,  R\otimes R \longrightarrow R$ be the multiplication map and $J=\Ker(\mu)$.
\begin{enumerate}
\item[(a)] The finitely generated $R$-module $\Omega^1_{R/K} = J/J^2$ is called the
{\bf K\"ahler differential module} (or the {\bf module of K\"ahler differentials})
of~$R/K$.

\item[(b)] For every $m\ge 0$, the exterior power
$\Omega^m_{R/K} = \Lambda_R^m \, \Omega^1_{R/K}$ is called the {\bf module of K\"ahler 
differential $m$-forms} of~$R/K$. Note that $\Omega^0_{R/K}=R$.

\item[(c)] The exterior algebra $\Omega^\dot_{R/K} = \bigoplus_{m\ge 0}\Omega^m_{R/K}$ 
is called the {\bf K\"ahler differential algebra} of~$R/K$.
\end{enumerate}
\end{definition}

Unless explicitly stated otherwise, all algebras will be $K$-algebras. Therefore we
usually simplify the notation and write $\Omega^m_R$ instead of $\Omega^m_{R/K}$ for all $m\ge 0$,
as well as $\Omega^\dot_R$ instead of $\Omega^\dot_{R/K}$.
 
The module of K\"ahler differentials of~$R/K$ comes together with the 
{\bf universal derivation} $d_R:\, R \longrightarrow \Omega^1_R$
which is given by $d_R(a) = a \otimes 1 - 1\otimes a + J^2$ for $a\in R$.
If no confusion can arise, we simply write~$d$ for this map.

Recall that $\Omega^\dot_{R/K}$ is the universal differential graded algebra 
of the algebra $R/K$ in the sense of \cite[Def.~3.3]{Kun1986}.  
In particular, the universal derivation extends to the modules of
K\"ahler $m$-forms and yields $K$-linear maps
$$
\delta_m:\; \Omega^m_R \,\To\, \Omega^{m+1}_R \hbox{\;\rm such that\;}
\delta_m (a_0\, da_1 \wedge \cdots \wedge da_m) = da_0 \wedge da_1 \wedge 
\cdots \wedge da_m 
$$  
for $m\ge 0$ and $a_0,\dots,a_m\in R$. Here we have $\delta_0 = d_R$.
Together with the map $\delta^\dot_R = \bigoplus_{m\ge 0} \delta_m$, the $R$-module
$\Omega^\dot_R$ is a {\bf differential graded algebra} (or simply {\bf DG algebra})
and can therefore be viewed as a complex.

\begin{definition}\label{def:DeRham}
Let $R$ be a finitely generated $K$-algebra as above.
\begin{enumerate}
\item[(a)] The complex
$$
0 \;\longrightarrow\; R \;\TTo{d_R}\; \Omega^1_R \;\TTo{\delta_1}\; \Omega^2_R \;\TTo{\delta_2} \;\cdots\; 
\TTo{\delta_{n-1}}\;  \Omega^n_R \;\To\; 0
$$
is called the {\bf de Rham complex} of $R/K$.

\item[(b)] The graded subring $Z(\Omega^\dot_R) = \Ker(\delta^\dot_R)$ of $\Omega^\dot_R$
is called the set of {\bf closed K\"ahler differential $m$-forms}.
  
\item[(c)] The set $B(\Omega^\dot_R) = \Im(\delta^\dot_R)$ is a two-sided graded ideal
in $Z(\Omega^\dot_R)$ (cf.~\cite[Rules 2.2]{Kun1986})
and is called the set of {\bf exact K\"ahler differential $m$-forms}.
  
\item[(d)] The graded $K$-algebra 
$$
\HdR(R/K) \;=\; {\textstyle\bigoplus_{m\ge 0}} \HdR^m(R) \;=\; 
Z(\Omega^\dot_R) / B(\Omega^\dot_R) \;=\; \Ker(\delta^\dot_R) / \Im(\delta^\dot_R)
$$
is called the {\bf de Rham cohomology} of $R/K$. We usually denote it simply by $\HdR(R)$.

\end{enumerate}
\end{definition}
  
For $m\ge 0$, the $m$-th homogeneous component $\HdR^m(R) = Z^m(\Omega^\dot_R) / B^{m-1}(\Omega^\dot_R)
\allowbreak =  \Ker(\delta_m) / \Im(\delta_{m-1})$ is called the {\bf $m$-th de Rham cohomology} of~$R/K$.
The de Rham cohomology of polynomial rings can be described as follows.

\begin{proposition}[De Rham Cohomology of Polynomial Rings]\label{prop:dRpoly}$\mathstrut$\\
Let $P=K[x_1,\dots,x_n]$.
\begin{enumerate}
\item[(a)][Poincar\'e's Lemma]\;
Let $\charac(K)=0$. Then the de Rham cohomology of~$P$ is given by
$\HdR^0(P)=K$ and $\HdR^m(P)=0$ for $m\ge 1$.

\item[(b)] Let $\charac(K)=p>0$. Then we have $\HdR^0(P) = \Ker(d_P) = K[x_1^p, \dots,x_n^p]$.

\item[(c)] Let $\charac(K)=p>0$ and $m\ge 0$. Then the $K$-linear map
$\Phi:\; \Omega^m_P \longrightarrow \HdR^m(P)$ defined by
$$
\Phi(a_0\, da_1 \wedge \cdots \wedge da_m) = a_0^p\, (a_1^{p-1}\, da_1) \wedge \cdots \wedge
(a_m^{p-1}\, da_m) + B^{m-1}(\Omega^\dot_P)
$$ 
is an isomorphism.

\end{enumerate}
\end{proposition}

\begin{proof}
The algebraic version of Poincar\'e's Lemma is, for instance, shown in~\cite[Prop.~7.1]{Har1975}.
Of course, claim~(b) is a special case of~(c). A direct proof is contained in the proof 
of~\cite[Prop.~5.6]{Kun1986}. Finally, in~\cite[Thm.~1.3.4]{BK} it is shown that $\id_{\overline{K}}\otimes_K \Phi$
is an isomorphism, where $\overline{K}$ is the algebraic closure of~$K$.
As $\overline{K}/K$ is faithfully flat, it follows that~$\Phi$ is an isomorphism, as well.
\end{proof}

An easy consequence of this proposition is that in characteristic zero, the highest de Rham
cohomology of an affine algebra vanishes.

\begin{corollary}\label{cor:highestDR}
Assume that $\charac(K)=0$ and $R=P/I$ with an ideal $I\subseteq P$.
Then we have $\HdR^n(R)=0$.
\end{corollary}

\begin{proof}
In the commutative diagram
$$
\begin{matrix}
\Omega^{n-1}_P & \longrightarrow & \Omega^n_P & \longrightarrow & 0\\
\downarrow && \;\downarrow\epsilon && \\
\Omega^{n-1}_R & \longrightarrow & \Omega^n_R && 
\end{matrix}
$$
the map $\epsilon:\; \Omega^n_P = P dx_1 \wedge \cdots \wedge dx_n \longrightarrow
\Omega^n_R = (P/( I + \langle \frac{\partial f}{\partial x_i} \mid f\in I \rangle))\, d\bar{x}_1
\wedge \cdots \wedge d\bar{x}_n$ is surjective. By Poincar\'e's Lemma, the map in the first row is surjective.
Therefore also the map in the second row is surjective, i.e., we have $\HdR^n(R) = 0$.
\end{proof}

The functorial properties of K\"ahler differential algebras imply that the de Rham cohomology
has good functorial properties. In particular, we are going to use the following results later on.

\begin{proposition}\label{prop:functorial}
Let $R$ be a finitely generated $K$-algebra as above.
\begin{enumerate}
\item[(a)] Given a field extension $K\subseteq L$, we have $\HdR(L\otimes_K R) \cong
L \otimes_K \HdR(R)$.

\item[(b)] Suppose that $A_1, \dots, A_s$ are finitely generated $K$-algebras such that
$R = A_1 \times \cdots \times A_s$. Then there is an isomorphism of graded $K$-algebras
$$
\HdR(R) \cong \HdR(A_1) \times \cdots \times \HdR(A_s).
$$  
\end{enumerate}
\end{proposition}

\begin{proof}
To show~(a), we note that $\Omega^\dot_{L\otimes_K R/L} \cong L\otimes_K  
\Omega^\dot_{R/K}$ and that $\delta^\dot_{L\otimes_K R}$ corresponds to
$\id_L\otimes_K \delta_R$ under this isomorphism (cf.~\cite[Cor.~4.3]{Kun1986}).
Thus the claim follows from the fact that the base change $K\subseteq L$ is flat.

Claim~(b) follows from $\Omega^\dot_R \cong \Omega^\dot_{A_1} \otimes \cdots \otimes
\Omega^\dot_{A_s}$ and the corresponding decomposition of~$\delta^\dot_R$ 
(cf.~\cite[Prop.~4.7]{Kun1986}).  
\end{proof}

Next we consider the case when $R = \bigoplus_{i\ge 0} R_i$ is a finitely generated,
positively graded $K$-algebra. This means that there is a presentation $R=P/I$,
where $P$ is positively graded by $\deg(x_i)\in \NN_+$ for $i=1,\dots,n$, 
and where~$I$ is a homogeneous ideal in~$P$. Ideals which are homogeneous with respect to a positive
$\ZZ$-grading on~$P$ are sometimes called {\bf quasi-homogeneous}.
For simplicity, we denote the residue classes of $x_1,\dots,x_n$ by $x_1,\dots,x_n$ again, 
if no confusion arises. In this setting, a useful tool to study the de Rham cohomology is the following complex.

\begin{definition}[The Euler-Koszul Complex]\label{def:Euler}\\
Let $R= \bigoplus_{i\ge 0} R_i$ be a finitely generated, positively graded $K$-algebra as above, 
and let $\m = \langle x_1,\dots,x_n\rangle$
be the ideal generated by the residue classes of the indeterminates of~$P$ in~$R$.
\begin{enumerate} 
\item[(a)] The {\bf Euler derivation} $e_R:\; R \longrightarrow R$ given by $e_R(a) = ia$ for $a\in R_i$
induces via the universal property of $\Omega^1_R$ an $R$-linear map $\epsilon_R:\; \Omega^1_R \longrightarrow \m$
which satisfies $\epsilon_R(dx_i)=x_i$ for $i=1,\dots,n$. This map is called the {\bf Euler form} 
on~$\Omega^1_R$.

\item[(b)] The Koszul complex of the Euler form is given by
$$ 
0 \;\longrightarrow\; \Omega^n_R \;\TTo{\epsilon_n}\;
\cdots \; \TTo{\epsilon_3}\;  \Omega^2_R  \; \TTo{\epsilon_2}\;
\Omega^1_R \;\TTo{\epsilon_1}\; \m \;\To\; 0 
$$
where $\epsilon_1 = \epsilon_R$ and $\epsilon_i$ satisfies
$$
\epsilon_i(w_1 \wedge \cdots \wedge w_i) = \tsum_{j=1}^i
(-1)^{j-1} \epsilon(w_j)\, w_1 \wedge \cdots \wedge \widehat{w_j} \wedge
\cdots \wedge w_i
$$
for $i\ge 2$ and $w_1,\dots,w_i \in \Omega^1_R$. 
This complex is called the {\bf Euler-Koszul complex} of~$R$.
The map $\epsilon = \bigoplus_i \epsilon_i$ is called the {\bf Euler contraction}
or the {\bf interior product} with the Euler vector field in $\Omega^\dot_R$.

\end{enumerate}
\end{definition}

If the characteristic of~$K$ is zero, the de Rham cohomology of positively graded $K$-algebras
is known to be trivial (see for instance~\cite{Wei}, Cor.~9.9.3). The following {\it folklore}
proposition provides a more detailed picture.

\begin{proposition}[De Rham Cohomology of Positively Graded Rings]\label{prop:dRgraded}$\mathstrut$\\
Let $R=P/I$ be a positively graded $K$-algebra, where $P=K[x_1,\dots,x_n]$ is positively graded by $\deg(x_i)
\in \NN_+$ for $i=1,\dots,n$, and where~$I$ is a homogeneous ideal in~$P$.
\begin{enumerate}  
\item[(a)] For $m\ge 1$ and a homogeneous element $w\in \Omega^m_R$, we have
$$
  (\delta_{m-1}\circ\epsilon_m + \epsilon_{m+1}\circ\delta_m)(w) = \deg(w)\cdot w
$$
This is occasionally called {\bf Cartan's magic formula}.
  
\item[(b)] If $\charac(K)=0$ then $\HdR^0(R)=K$ and $\HdR^m(R)=0$ for $m\ge 1$.
  
\item[(c)] If $\charac(K)= p > 0$ and $p\nmid i$ then $\HdR^m(R)_i=0$ for all $m\ge 1$.
 
\end{enumerate}
\end{proposition}

\begin{proof}
The proof of~(a) is well-known. For the convenience 
of the readers, we work out the details. Let $m\in \{1,\dots,n\}$, and let $w\in \Omega^m_R$
be a homogeneous element of degree $m+k$ of the form $w= f\, dx_{i_1} \wedge \cdots \wedge dx_{i_m}$,
where $f\in R_k$ and $1\le i_1 < \cdots < i_m \le n$. 
Using $\epsilon(df)=e_R(f) = k\,f$ and~\cite[Prop.~1.6.2.d]{BH1993}, we calculate
\begin{align*}
(\epsilon_{m+1}\circ \delta_m)(w) &\;=\; \epsilon_{m+1}( df \wedge dx_{i_1} \wedge \cdots \wedge dx_{i_m} )\\
&\;=\; k\,f\, dx_{i_1}\wedge\cdots\wedge dx_{i_m} - df \wedge \epsilon_m(dx_{i_1}\wedge\cdots\wedge dx_{i_m})
\end{align*}
On the other hand, we have
\begin{align*}
(\delta_{m-1} \circ \epsilon_m)(w) &\;=\; \delta_{m-1}( f \cdot \tsum_{j=1}^m (-1)^{j-1}\, x_{i_j}\,
dx_{i_1} \wedge \cdots \wedge \widehat{dx_{i_j}} \wedge \cdots \wedge dx_{i_m} )\\
&\;=\; df \wedge \epsilon_m ( dx_{i_1} \wedge \cdots \wedge dx_{i_m} )
+ m\, f\, dx_{i_1} \wedge \cdots \wedge dx_{i_m}
\end{align*}
Altogether, we get $(\delta_{m-1}\circ\epsilon_m + \epsilon_{m+1}\circ\delta_m)(w) = (k+m)\,w = 
\deg(w)\cdot w$, as claimed.

To show~(b), we can use $\Omega^m_R=0$ for $m\ge n+1$ to conclude that
$\HdR^m(R)=0$ for $m\ge n+1$. For $m\in \{1,\dots,n\}$, we choose a homogeneous element
$w\in Z^m(\Omega^\dot_R)$. Then we let $w' = \frac{1}{\deg(w)}\, \epsilon_m(w) \in\Omega^{m-1}_R$.
Using~(a), we get $w = \delta_{m-1}(w') \in B^m(\Omega^\dot_R)$. This shows $\HdR^m(R)=0$.
Finally, we note that $\HdR^0(R) = \Ker(d_R) = K$, since $df=0$ implies $\epsilon(df)=k\,f=0$, and therefore
$f=0$ or $k=0$ for every $f\in R_k$.

It remains to prove~(c). As in the proof of~(b), it suffices to consider the case $m\in \{1,\dots,n\}$.
Let $w\in Z^m(\Omega^\dot_R)$ be a homogeneous element of degree~$k$ such that $p\nmid k$.
Then the element $w' = \frac{1}{k}\, \epsilon_m(w)$ satisfies $\delta_{m-1}(w')=w$, and therefore
we obtain $\HdR^m(R)=0$.
\end{proof}

In fact, we can weaken the assumptions of this proposition to some extent, since it suffices that
the ring~$R$ is a \textbf {quasi-homogeneous ring}, i.e., isomorphic to a positively graded $K$-algebra 
as in the proposition. Let us formulate this as an explicit corollary, since we will use it later.

\begin{corollary}[De Rham Cohomology of Quasi-Homogeneous Rings]\label{cor:HdRQuasiHomog}$\mathstrut$\\
Let $P=K[x_1,\dots,x_n]$ be positively graded by $\deg(x_i) \in \NN_+$ for $i=1,\dots,n$,
let~$I$ be a homogeneous ideal in~$P$, and let $\phi:\; P \longrightarrow P$ be a $K$-algebra
homomorphism. If $\charac(K)=0$ then the de Rham cohomology of the ring $R=P/\phi(I)$
satisfies $\HdR^0(R)=K$ and $\HdR^m(R)=0$ for $m\ge 1$.
\end{corollary}

Now we return to the general setting. So, let $R=P/I$ be again an arbitrary finitely generated
$K$-algebra, where $P=K[x_1,\dots,x_n]$ and~$I$ is an ideal in~$P$.

\begin{proposition}\label{prop:diagram}
Let $I \subseteq J$ be two ideals in~$P$, let $R=P/I$, let $S=P/J$, let
$\phi:\; R \longrightarrow S$ be the canonical epimorphism, let $\bar{J}$ be the residue class ideal
of~$J$ in~$R$, and let $m\ge 1$.
\begin{enumerate}
\item[(a)] There is a commutative diagram with exact rows
\begin{alignat*}{6}
0 \;\To\;\;& \bar{J} \Omega^m_R     &{}+{}& d\bar{J} \wedge \Omega^{m-1}_R &\;\To\; 
        &\;\Omega^m_R &&\;\TTo{\phi_m}\; &&\Omega^m_S &&\;\To\; 0\\
&& \downarrow\;&\hat{\delta}_m && \;\downarrow\delta_m  && && \downarrow\bar{\delta}_m &&\\
0 \;\To\;\;& \bar{J} \Omega^{m+1}_R &{}+{}& d\bar{J} \wedge \Omega^m_R &\;\To\; 
        &\;\Omega^{m+1}_R &&\;\TTo{\phi_{m+1}}\; &&\Omega^{m+1}_S &&\;\To\; 0\\
\end{alignat*}

\item[(b)] We have $\bar{J} \Omega^m_R + d\bar{J} \wedge \Omega^{m-1}_R = d\bar{J} \wedge \Omega^{m-1}_R$.

\item[(c)] We have $\phi_{m+1}(\Im(\delta_m)) = \Im(\bar{\delta}_m)$.

\item[(d)] Assume that $\HdR^m(R)=0$. Then we have $\HdR^m(S)=0$ if and only if the induced map
$\Coker(\hat{\delta}_m) \longrightarrow \Coker(\delta_m)$ is injective.

\end{enumerate}
\end{proposition}

\begin{proof}
First we show~(a).
To construct the two exact sequences in~(a), we use~\cite[Prop.~4.12]{Kun1986}.
The maps $\delta_m$ and $\bar{\delta}_m$ are the $m$-th components of the universal differentials
of~$\Omega^\dot_R$ and $\Omega^\dot_S$, respectively, whence the square on the right commutes.
The map $\hat{\delta}_m$ is induced by~$\delta_m$, and thus the square on the left commutes as well.

Next we prove~(b). The $R$-module $\Omega^m_R$ is generated by the elements $da_{i_1} \wedge \cdots\wedge da_{i_m}$
with $a_{i_1},\dots,a_{i_m}\in R$. For $f\in \bar{J}$, we calculate
$$
f da_{i_1} \wedge \cdots \wedge fda_{i_m} \;=\;  d(fa_{i_1}) \wedge da_{i_2} \wedge \cdots \wedge da_{i_m} -
a_{i_1}\, df \wedge da_{i_2} \wedge \cdots \wedge da_{i_m} 
$$
and this proves the claim.

For the proof of~(c), we first choose $w \in \Im(\delta_m)$ and write $w = \delta_m(w')$
with $w'\in \Omega^m_R$. Then $\phi_{m+1}(w) = \bar{\delta}_m ( \phi_m (w')) \in \Im(\bar{\delta}_m)$
proves the inclusion ``$\subseteq$''. Conversely, let $w \in \Im(\bar{\delta}_m)$ and write 
$w = \bar{\delta}_m(w')$ with $w'\in\Omega^m_S$. Since $\phi_m$ is surjective, we find $w'' \in \Omega^m_R$
such that $w' = \phi_m(w'')$. Then $w = \bar{\delta}_m (\phi_m(w'')) = \phi_{m+1} (\delta_m(w'')) 
\in \phi_{m+1} (\Im(\delta_m))$ proves the claim.

It remains to prove~(d). Suppose that $\HdR^m(S)=0$ and choose $w \in \bar{J} \Omega^{m+1}_R + 
d\bar{J} \wedge \Omega^m_R$ such that the image of $w + \Im(\hat{\delta}_m)$ in $\Coker(\delta_m)$
is zero. By~(b), we may assume that $w\in d\bar{J} \wedge \Omega^m_R$. Since the residue class of~$w$
in $\Coker(\delta_m)$ is zero, we find an element $w'\in \Omega^m_R$ such that $w= \delta_m(w')$.
The fact that $\phi_m(w) \in \Ker(\bar{\delta}_m)$ and the hypothesis $\HdR^m(S)=0$ imply that there
exists an element $v\in \Omega^{m-1}_S$ such that $\bar{\delta}_{m-1}(v)= \phi_m(w)$.
Hence we find an element $v' \in \Omega^{m-1}_R$ with $\phi_{m-1}(v')=v$. Then 
$w'' = \delta_{m-1}(w'-v')$ satisfies $\phi_m(w'') = \phi_m(w) - \phi_m(\delta_{m-1}(v')) = 
\phi_m(w) - \bar{\delta}_{m-1}(v) = 0$. Consequently, we have $w'' \in d\bar{J}\wedge \Omega^{m-1}_R$
and $\hat{\delta}_m(w'') = \delta_m(w'') = w$. Thus the residue class of~$w$ in $\Coker(\hat{\delta}_m)$
is zero.

Conversely, assume that the map $\imath:\; \Coker(\hat{\delta}_m) \longrightarrow \Coker(\delta_m)$ is injective
and choose $w\in \Ker(\bar{\delta}_m)$. We have to show that $w\in \Im(\bar{\delta}_{m-1})$.
Let $\eta:\; \Ker(\bar{\delta}_m) \longrightarrow \Coker(\hat{\delta}_m)$ be the connecting 
homomorphism of the Snake Lemma. Then $\eta(w) \in \Ker(\imath)$ implies $\eta(w)=0$, and hence 
$w = \phi_m(w')$ with $w'\in \Ker(\delta_m)$. Next, the hypothesis $\HdR^m(R)=0$ yields an element
$w''\in \Omega^{m-1}_R$ such that $w' = \delta_{m-1}(w'')$. Consequently, we get
$w = \phi_m(\delta_{m-1}(w'')) = \bar{\delta}_{m-1}(\phi_{m-1}(w'')) \in \Im(\bar{\delta}_{m-1})$,
as we wanted to show. 
\end{proof}

Recall that~$P$ is equipped with the {\bf degree filtration} $\cD = (\cD_i)_{i \in\ZZ}$
which satisfies $\cD_i = P_{\le i} = \{ f\in P \mid \deg(f) \le i\}$ for all $i\in \ZZ$.
The {\bf degree form} $\DF(f)$ of a non-zero polynomial $f\in P$ is its initial form with respect to this filtration,
i.e., the homogeneous component of highest degree of~$f$. For a $P$-module~$M$ which is
a residue class module of a graded free $P$-module $F = \bigoplus_{i=1}^r P(-k_i)$ by a (not necessarily graded)
submodule~$U$, we define a {\bf degree filtration} $\cE = (\cE_i)_{i\in\ZZ}$
on $M = F/U$ by letting $\cE_i = \{ (\bar{g}_1,\dots,\bar{g}_r) \mid \deg(g_j) + k_j \le i \}$ for $j=1,\dots,r$.
Again we let the {\bf degree form} $\DF(m)$ of a non-zero element of~$M$ be its homogeneous component of
highest degree. For simplicity, we also set $\DF(0)=0$.
  
Moreover, a set of non-zero  elements $\{m_1,\dots,m_s\}$ of~$M$ is called a {\bf Macaulay basis} 
of~$M$ if their degree forms generate the degree form module $\DF(M)= \langle \DF(m) \mid m\in M\rangle$ 
(see~\cite[Sect.~4.2.b]{KR2005}). The next theorem establishes a connection between 
special Macaulay bases of a certain $P$-module and the vanishing of the de Rham cohomology of~$R$.

\begin{theorem}[Macaulay Bases and the de Rham Cohomology]\label{thm:MacBasisDR}$\mathstrut$\\
Let $m\ge 1$, let $I$ be an ideal of~$P$, and let $R=P/I$. 
Assume that $\charac(K)=0$ and that the submodule $dI \wedge \Omega^m_P$ of the free $P$-module 
$\Omega^{m+1}_P$ has a Macaulay basis of the form $\{ df_1\wedge w_1, \dots, df_r \wedge w_r\}$
with $f_i \in I$ and $w_i\in \Omega^m_P$. Then $\HdR^m(R)=0$.
\end{theorem}

\begin{proof}
For a contradiction, assume that $\HdR^m(R)_\ell \ne 0$ for some $\ell\in \ZZ$. 
We note that $\ell\ge 1$ and choose~$\ell$ minimally. By Prop.~\ref{prop:diagram}.d, there exists an element 
$w\in \Im(\delta_m)\cap (dI \wedge \Omega^m_P)$ of degree~$\ell$ such that $w \notin \Im(\hat{\delta}_m)$.
By the assumption and~\cite[Tut.~48.c]{KR2005}, we can write 
$$
w = h_1 ( df_1\wedge w_1) + \cdots + h_r( df_r \wedge w_r)
=  df_1\wedge h_1 w_1 +\cdots + df_r \wedge h_r w_r
$$
with $h_i\in P$ such that $\deg( df_i \wedge h_i w_i)\le \ell$ for $i=1,\dots,r$.
Without loss of generality, let $k\le r$ be such that $df_1 \wedge h_1 w_1, \dots, df_k \wedge h_k w_k$
are the summands of degree~$\ell$.

Then we have $\DF(w)= d(\DF(f_1))\wedge \DF(h_1w_1) + \cdots + d(\DF(f_k)) \wedge \DF(h_k w_k)$.
For simplicity, we write $g_i=\DF(f_i)$ and $\tilde{w}_i = \DF(g_iw_i)$ for $i=1,\dots,k$.
Now we consider the homogeneous ideal $J = \langle g_1, \dots, g_k \rangle$
and the commutative diagram
\begin{alignat*}{6}
0 \;\To\;\;& J \Omega^m_P     &{}+{}& dJ \wedge \Omega^{m-1}_P &\;\To\; 
        &\;\Omega^m_P &&\;\TTo{\phi_m}\; &&\Omega^m_{P/J} &&\;\To\; 0\\
&& \downarrow\;&\hat{\delta}_m && \;\downarrow\delta_m  && && \downarrow\bar{\delta}_m &&\\
0 \;\To\;\;& J \Omega^{m+1}_P &{}+{}& dJ \wedge \Omega^m_P &\;\To\; 
        &\;\Omega^{m+1}_P &&\;\TTo{\phi_{m+1}}\; &&\Omega^{m+1}_{P/J} &&\;\To\; 0\\
\end{alignat*}
which corresponds to the diagram in Proposition~\ref{prop:diagram}.a with $R=P$ and $S=P/J$.

Since $w\in \Im(\delta_m) = \Ker(\delta_{m+1})$ and $\delta_{m+1}$ is a homogeneous map, 
this implies $\delta_{m+1} (\DF(w))=0$. By Proposition~\ref{prop:dRgraded}, we know $\HdR^m(P/J)_\ell = 0$. 
Hence an application of Prop.~\ref{prop:diagram}.d yields $\DF(w) \in \Im(\hat{\delta}_m)$.
Thus there exist homogeneous elements $u_1,\dots,u_k \in \Omega^m_P$ and 
$v_1,\dots,v_k \in \Omega^{m-1}_P$ such that the element
$$
w' \;=\; \tsum_{i=1}^k \, g_i u_i + 
\tsum_{i=1}^k \, dg_i \wedge v_i \in dJ \wedge \Omega^{m-1}_P
$$
satisfies
$$
\DF(w) \;=\; \hat{\delta}_m (w')   \;=\; \tsum_{i=1}^k (dg_i\wedge u_i + g_i\, du_i) + 
\tsum_{i=1}^k \, dg_i\wedge d v_i
$$
Now we consider the element
$\tilde{w} = w - \sum_{j=1}^k df_j \wedge (u_j +dv_j) - \sum_{j=1}^k f_j\, du_j$. 
Notice that the degree form of~$w$ cancels here, and hence $\deg(\tilde{w}) < \deg(w)$.
Moreover, we have $df_j \wedge dv_j = \hat{\delta}_m (f_j dv_j)\in \Im(\hat{\delta}_m)$ as well as 
$df_j\, u_j + f_j du_j = \hat{\delta}_m(f_j u_j) \in \Im(\hat{\delta}_m)$ for $i=1,\dots,k$,
so that $\tilde{w} \in \Im(\delta_m) \cap dJ \wedge \Omega^m_R$ has the same residue class in $\Coker(\hat{\delta}_m)$
as~$w$. By the minimality of $\deg(w)$, it follows that this residue class is zero, a contradiction.
\end{proof}

The hypothesis in this theorem that $dI \wedge \Omega^m_P$ has a Macaulay basis
of the required special form is essential, as the following example shows. It is inspired 
by~\cite[Ex.~12]{Sch2012}.

\begin{example}
Let $P = \QQ[x,y]$, let $I = \langle f\rangle$ with $f\in P \setminus \QQ$,
and let $R=P/I$. 
\begin{enumerate}
\item[(a)] For $f = xy-x$, we see that $df = (y-1)\, dx + x\, dy$.
Recall that we have $\Omega^1_R = \Omega^1_P/(dI+I\Omega^1_P)$. 

\begin{itemize}
\item[(1)] To compute $\HdR^0(R)=\Ker(d_R)$, we need the kernel of the $K$-linear map
$d_R:\; P / \langle xy -x\rangle  \longrightarrow (P dx \oplus P dy) / 
(dI + I \Omega^1_P)$.
The leading term module of $dI+I\Omega^1_P$ with respect to the module term ordering
{\tt DegRevLex-Pos} (see~\cite[Ex.~1.4.16]{KR2000}) is $\{ x\,dy,\, xy\,dx,\,
y^2\, dx\}$, and the normal form of every element in $Pdx \oplus Pdy$ is in 
$\QQ\, y\,dx \cup \QQ[x]\,dx \cup \QQ[y]\,dy$.
From this it follows that $\Ker(d_R) = \QQ$.

\item[(2)] In order to prove $\HdR^1(R)=0$, we use the theorem. The $P$-submodule
$$
\qquad\qquad dI \wedge \Omega^1_P = \langle df\wedge dx,\, df\wedge dy,\,  f\, dx\wedge dy \rangle
= \langle df\wedge dx,\, df\wedge dy \rangle
$$
of~$\Omega^2_P$ has the Macaulay basis $\{ df\wedge dx,\, df\wedge dy\}$, since
$f\,dx\wedge dy = x\, df\wedge dy$. Thus the theorem yields $\HdR^1(R)=0$.

\item[(3)] Finally, $\HdR^2(R)=0$ follows from Corollary~\ref{cor:highestDR}.
\end{itemize}

\item[(b)] Now we change this example slightly by setting $f=x y^2 - x - 1$.
Proceeding as in~(a), we check that $\HdR^0(R)=\QQ$ and $\HdR^2(R)=0$.
However, when we compute $dI \wedge \Omega^1_P$, we get
\begin{align*}
\qquad\qquad &dI \wedge \Omega^1_P \;=\; \langle df \wedge dx,\, df\wedge dy,\, f\, dx\wedge dy\rangle & \\
&\;=\; \langle -2xy\, dx\wedge dy,\, (y^2-1)\, dx\wedge dy,\, (xy^2-x-1)\, dx\wedge dy\rangle & \\
&\;=\; \langle dx \wedge dy\rangle &
\end{align*}
Clearly, the element $dx\wedge dy$ is not of the form $dg\wedge w$ with $g\in I$ and 
$w\in \Omega^1_P$. Therefore we are not in the setting of the theorem.
A direct calculation of $\HdR^1(R)$ yields the 2-dimensional $\QQ$-vector space
generated by the residue classes of $x\, dy$ and~$xy\, dy$.
\end{enumerate}
\end{example}

In the case of positive characteristic $\charac(K)=p>0$, the theorem does not hold
either, as our next example shows.

\begin{example}
Let $K=\mathbb{F}_3$, let $P=K[x,y]$, let $I=\langle f_1,\, f_2\rangle$, where $f_1=x^4 + y^3$
and $f_2 = y^4 + x^3$, and let $R=P/I$. For every degree compatible term ordering~$\sigma$ we have
$\LT_\sigma(I) = \langle x^4, y^4\rangle$. Hence $R$ is a 0-dimensional $K$-algebra.

Moreover, for $m=1$, the elements $df_1\wedge dy = x^3\, dx\wedge dy$ and $df_2\wedge dx = 
-y^3\, dx\wedge dy$ form a Macaulay basis of $dI + I\Omega^1_P$, as the degree forms of
$f_1\, dx\wedge dy$ and $f_2\, dx\wedge dy$ are clearly multiples of these homogeneous elements.
Hence the hypothesis of the theorem is satisfied.

However, it is not difficult to verify that the image of the universal derivation $d_R$ is the 8-dimensional
vector space 
$$
U = \langle dx, dy, xdx, ydy, ydx + xdy, y^2dx - x y dy, 
x ydx - ydy, x y^2dx + x^2 ydy \rangle_K,
$$
and the kernel of $\delta_1$ is the 10-dimensional vector space $U + \langle x^2 dx,\, y^2 dy \rangle_K$.
Consequently, we have $\dim_K(\HdR^1(R))=2$.
\end{example}

In general, we cannot expect $\HdR^m(R)=0$ for $m\ge 1$ if~$R$ is not quasi-homogeneous.
In the next two sections we study the dimensions of $\HdR^m(R)$ for $m\ge 0$
if $\dim(R)=0$, i.e., if~$R$ is the affine coordinate ring of a 0-dimensional scheme.

\bigbreak
%
%

\section{The De Rham Cohomology of Artinian Local Rings}\label{sec3}

One of the simplest cases of a 0-dimensional scheme is a (not necessarily reduced) $K$-rational
point. In this section we let $K$ be a perfect field and $A=P/I$ an Artinian local ring, 
where $P=K[x_1,\dots,x_n]$, and where $I\subseteq P$ is an ideal which is $\M$-primary
for $\M = \langle x_1,\dots,x_n\rangle$. Our goal is to describe the de Rham cohomology of~$A$
as explicitly as possible.

\medskip
\subsection{Reduced and Fat Points.}
The next case we look at is the case of quasi-homo\-ge\-neous ideals~$I$.
In this case, Prop.~\ref{prop:dRgraded} shows that, in characteristic zero,
the de Rham cohomology of~$A$ is concentrated in $\HdR^0(A)=K$.
In particular, this applies to a reduced point, where $I=\M$, and
to the case of a $K$-rational {\bf fat point} of multiplicity $\nu \ge 1$, i.e.,
to the case $I=\M^\nu$. In these cases the next proposition provides a somewhat 
more detailed picture.

\begin{proposition}\label{prop:dRFatPoint}
Let $\nu\ge 1$, and let $A = P/\M^\nu$.
\begin{enumerate}
\item[(a)] If $\charac(K)=0$ then $\HdR^0(A)=K$ and $\HdR^m(A)=0$ for $m\ge 1$.

\item[(b)] If $\charac(K)=p>0$ then $\HdR^0(A) = A^p$ and $\HdR^m(A)_i=0$
whenever $i\ge 0$ satisfies $p\nmid i$ or $i\ge \nu + m$, or if $m>n$.

\item[(c)] If $\charac(K)=p \ge \nu +n$ then $\HdR^m(A)=0$ for all $m\ge 1$.

\end{enumerate}
\end{proposition}

\begin{proof}
Claim~(a) follows from Prop.~\ref{prop:dRgraded}.b, as $A$ is a standard graded $K$-algebra.

To prove the first part of~(b), we note that $A_i=P_i$ for $0\le i\le\nu-1$ and $A_i=0$ for $i\ge \nu$.
Thus the kernel of~$d_A$ agrees with the kernel of~$d_P$ in degrees $\le \nu-1$ and is zero in higher degrees.
The kernel of~$d_P$ equals $P^p=K[x_1^p, \dots, x_n^p]$ by Prop.~\ref{prop:dRpoly}.b, and this implies the claim.

To compute the higher de Rham cohomology groups in~(b), we use Prop.~\ref{prop:dRgraded}.c and
the fact that $(\Omega^m_A)_i=0$ for $i\ge \nu+m$.

Finally, we note that~(c) follows from~(b), since $p\nmid i$ for all $i< \nu + m \le \nu+n$
whenever $m\in \{1,\dots,n\}$.
\end{proof}

\bigskip
\subsection{Local Complete Intersections.}
If the Artinian local ring $A = P/I$ is a complete intersection, i.e., if~$I$ is generated
by a regular sequence contained in~$\M = \langle x_1,\dots,x_n\rangle$, 
the higher de Rham cohomology groups of~$A$ may not vanish, even 
if $\charac(K)=0$. The following example is manufactured after an example in~\cite{Rei}.
Recall that the {\bf Euler characteristic} of the de Rham complex is 
$$
\chi(\Omega^\bullet_{A/K}) = \sum_{i\ge 0} (-1)^i \dim_K (\HdR^i(A))
$$ 
In this formula we have $\HdR^k(A)=0$ for $k\ge n = \dim(P)$ by Corollary~\ref{cor:highestDR}.

\begin{example}\label{ex:KerD3dim}
Let $A = \QQ[x,y]/ \langle f,g\rangle$, where
$f= x^4+y^5+xy^5$ and $g=x^6$. Then~$A$ is a 0-dimensional local complete intersection
and $\dim_\QQ(A)=30$.

Using ApCoCoA, we verify that the kernel of the universal derivation $d_A$ is the 3-dimensional
$\QQ$-vector space
$$
\Ker(d_A) \;=\; \langle 1,\; (-8/45)\, y^9 -y^8 +x^5 y^2 - x^4 y^3,\; (-37/45)\ y^9 + x^5 y^3 \rangle_\QQ
$$
Moreover, we can check that the image of~$d_A$ is 27-dimensional and the kernel of 
$\delta_1:\; \Omega^1_A \longrightarrow \Omega^2_A$ is 28-dimensional.
A concrete 1-form with a non-trivial de Rham cohomology class is for instance
$w=x^4y\, dx$. 

Altogether, in this example the Euler characteristic of the de Rham complex is $\chi(\Omega^\bullet_A) = 2$.
\end{example}

Our next example illustrates that the Euler characteristic of the de Rham complex of a complete intersection
can be larger than one, even if the higher de Rham cohomology vanishes.

\begin{example}\label{ex:EulerChar2}
In $P=\QQ[x,y]$, the polynomials $f_1 = x^4$ and $f_2 = y^5 + x^2y^2 +xy^3$ define a local Artinian complete
intersection $A = P/\langle f_1, f_2\rangle$. In fact, $\{f_1, f_2\}$ is a strict regular sequence and
$\dim_\QQ(A)=20$.

It is easy to check that $\Ker(d_A) = \QQ \oplus \QQ\,x^3 y^4$. Thus the image of $d_A$ is 18-dimensional.
A straightforward calculation of the kernel of $\delta_1:\; \Omega^1_A \longrightarrow \Omega^2_A$
shows that this kernel is also 18-dimensional. Hence we get $\HdR^1(A)=0$, and
Corollary~\ref{cor:highestDR} yields $\HdR^2(A)=0$.

Altogether, the Euler characteristic of the de Rham complex satisfies $\chi(\Omega^\bullet_A)=2$.
\end{example}

Notice that the ring~$A$ in this example is even a \textbf{strict complete intersection}, i.e.,
that the degree forms $\DF(f_1),\DF(f_2)$ of the polynomials $f_1, f_2$ form a homogeneous regular sequence.

Another attempt to generalize the reduced case to complete intersections uses the following definition
which was introduced in~\cite{Sal}.

\begin{definition}
Let $\M= \langle x_1,\dots,x_n\rangle$ be the homogeneous maximal ideal of~$P$, and let
$\cM= (\M^i)_{i\ge 0}$ be the $\M$-adic filtration of~$P$. 
For every $f\in P\setminus \{0\}$, the initial form $\ini_\cM(f)$
is the homogeneous component of lowest degree of~$f$.
\begin{enumerate}
\item[(a)] A tuple of polynomials $(f_1,\dots,f_\ell)\in P^\ell$ is called
a \textbf{super-regular sequence} if $(\ini_\cM(f_1),\dots,
\ini_\cM(f_\ell))$ is a homogeneous regular sequence in~$P$.

\item[(b)] A local $K$-algebra $A$ is called a \textbf{super-strict
complete intersection} (or, in short, an SSCI) if it is isomorphic to $P / \langle f_1,\dots,f_\ell\rangle$
with a super-regular sequence $(f_1,\dots,f_\ell)$.

\end{enumerate}
\end{definition}

Unfortunately, even a local Artinian super-strict complete intersection can have
non-trivial de Rham cohomology, as the next example shows.

\begin{example}\label{ex:badSSCI}
Let $P = \QQ[x,y]$ be standard graded.
\begin{enumerate}
\item[(a)] Consider the ideal
$I=\langle f_1,f_2\rangle$, where $f_1 = x^4+y^5+xy^4$ and $f_2=y^6$.
Then $A=P/I$ is a local Artinian ring with maximal ideal $\langle \bar{x},\bar{y}\rangle$
and $(f_1,f_2)$ is a super-regular sequence.
A computation with ApCoCoA shows $\Ker(d_A) = \langle 1,\, x^6y \rangle_\QQ$
and $\HdR^1(A)=\{0\}$. By Corollary~\ref{cor:highestDR}, we also have $\HdR^2(A) = \{0\}$.

\item[(b)] Now we look at the ideal $J = \langle f_1, f_3\rangle$, where $f_1= x^4+y^5+xy^4$
as in~(a) and $f_3 = y^8$. Also the ring $B=P/J$ is a local Artinian super-strict complete
intersection.

In this case we calculate $\Ker(d_A) = \langle 1,\, 29 x^6y^2 - 68 x^5 y^3,\, x^6 y^3 \rangle_\QQ$
and $\dim_K(\HdR^1(B)) = 1$. Here the residue class of $x^3 y^2\, dx$ is a $\QQ$-basis of $\HdR^1(B)$.

\end{enumerate}
\end{example}

So, if we have an ideal~$I$ which is not quasi-homogeneous, we have to check in some other way
whether the de Rham cohomology of~$P/I$ is trivial. The next theorem can help us. 
In order to prove it, we need some preparations.

Assume that we are in the setting above. In particular, let
$I=\langle f_1,\dots,f_n\rangle$ be an ideal generated by a super-regular sequence,
and let $A=P/I$.

The $\M$-adic filtration~$\cM$ on~$P$ induces the $\m$-adic filtration
$\cF = (\cF_i)_{i\in\ZZ}$ on~$A$, where $\cF_i = \m^i$. (Here we let $\cF_i=P$ for $i\le 0$.)
Since the ideal $\m$ is nilpotent, the filtration $\cF$ is finite, and thus the
graded ring $\gr_{\m}(A) = \bigoplus_{i\ge 0} \m^i / \m^{i+1}$ is a local Artinian
$K$-algebra, as well. This graded ring has the presentation
$$
\gr_\m(A) \;\cong\; P / \ini_\cM(I) \;=\;
P / \langle \ini_\cM(f_1),\dots, \ini_\cM(f_n) \rangle
$$
and is therefore a 0-dimensional homogeneous complete intersection.
Here the set $\{ \ini_\cM(f_1),\dots, \ini_\cM(f_n) \}$ generates the initial ideal
of~$I$ by~\cite[Prop.~2.1]{VV}. 

Next we define an $\M$-adic filtration on $\Omega^1_P$.
By letting $\cN_i = \M^{i-1} dx_1 \cup \cdots \cup \M^{i-1} dx_n$ for
$i\ge 1$ and $\cN_i=\Omega^1_P$ for $i\le 0$, we get a decreasing filtration $\cN = (\cN_i)_{i\in\ZZ}$
which turns $\Omega^1_P$ into a filtered module over~$(P,\cM)$. This
filtration induces a filtration $\cG^{(1)}$ on
$\Omega^1_A \cong \Omega^1_P / (I\Omega^1_P + dI)$.
As $\Omega^1_A$ is generated by $\{d\bar{x}_1,\dots,d\bar{x}_n\}$
and~$\m$ is nilpotent, the filtration $\cG^{(1)}$ is finite, too.

In a similar fashion we now turn $\Omega^m_P$ into a filtered $(P,\cM)$-module
for every $m\ge 1$. Given $i\ge m$, we define
$$
\cN^{(m)}_i \;=\; \bigoplus_{1\le j_1 < \cdots < j_m\le n} \M^{i-m}\,
dx_{j_1} \wedge \cdots \wedge dx_{j_m} 
$$
and for $i<m$ we let $\cN^{(m)}_i = \Omega^m_P$.
Thus we obtain a descending filtration on~$\Omega^m_P$. 
For simplicity, we denote it by~$\cN$ again.

Finally, we turn $\Omega^m_A$ for every $m\ge 1$
into a filtered $(A,\cF)$-module by letting 
$$
\cG^{(m)}_i \;=\; 
\bigoplus_{1\le j_1 < \cdots < j_m\le n} \m^{i-m}\,
d\bar{x}_{j_1} \wedge \cdots \wedge d\bar{x}_{j_m} 
$$
for $i\ge m$ and $\cG^{(m)}_i = \Omega^m_A$ for $i<m$. Again, the filtration
$\cG^{(m)} = (\cG^{(m)}_i)_{i\in\ZZ}$ is finite.
Altogether, we see that $\Omega^\dot_A = \bigoplus_{m\ge 0} \Omega^m_A$
is a filtered $(A,\cF)$-module and that its filtration $\cG$ is finite.

To simplify the notation, we frequently write $\gr(A)$ instead of $\gr_\m(A)$ and
$\gr(P)$ instead of $\gr_\M(P)$. Of course, the graded ring $\gr(P)$ is isomorphic to~$P$, 
but for the sake of clarity we shall keep the distinction. 
The next lemma plays a key role in the analysis of the de Rham cohomology 
using these filtrations.

\begin{lemma}\label{lem:Phi}
Let $I$ be an $\M$-primary ideal in $P=K[x_1,\dots,x_n]$, and let $A=P/I$.
Then there exists a surjective homomorphism of $\gr_\m(A)$-DG-algebras 
$$
\Phi:\; \Omega^\dot_{\gr_\m(A)} \;\longrightarrow\; \gr_{\cG} ( \Omega^\dot_A )
$$
which is defined by $\Phi(d\bar{x}_{i_1} \wedge \cdots \wedge d\bar{x}_{i_m}) = 
\overline{dx_{i_1} \wedge \cdots \wedge dx_{i_m}}$ for $i_1 < \cdots < i_m$.
\end{lemma}

\begin{proof}
To begin with, we construct~$\Phi$ as a $\gr(A)$-linear map. Then we verify its
other properties.

First we define a map 
$\gamma:\; \gr(A) \longrightarrow \gr_\cG(\Omega^1_A)$ as follows. For $i\ge 0$
and $\bar{a}\in \gr(A)_i$, we choose a representative $a\in \m^i$, and let $\gamma(\bar{a})
=d\, a + \cG^{(1)}_{i+1}$. It is easy to check that~$\gamma$ is a well-defined derivation of
$\gr(A)/K$. By the universal property of $\Omega^1_{\gr(A)}$, we obtain a homomorphism
of $\gr(A)$-modules
$$
\Phi^{(1)}:\; \Omega^1_{\gr(A)} \;\longrightarrow\; \gr_{\cG^{(1)}}(\Omega^1_A)
$$
which satisfies $\Phi^{(1)}(d\bar{x}_i) = dx_i + \cG^{(1)}_2$ for $i=1,\dots,n$.

Next we apply the functor $\Lambda^m_{\gr(A)}$ and get a $\gr(A)$-linear map
$$
\Phi^{(m)}:\; \Omega^m_{\gr(A)} \;\longrightarrow\; \Lambda^m_{\gr(A)} \, \gr_{\cG^{(1)}}(\Omega^1_A)
$$
Then we define the multiplication map $\mu_m:\ \Lambda^m_{\gr(A)} \gr_{\cG^{(1)}}(\Omega^1_A)
\longrightarrow \gr_{\cG^{(m)}}(\Omega^m_A)$ by $\mu_m(\overline{da_1} \wedge \cdots \wedge 
\overline{da_m}) = \overline{da_1\wedge \cdots \wedge da_m}$ and check that it is well-defined
and $\gr(A)$-linear. Altogether, we have constructed the desired $\gr(A)$-linear map
$$
\Phi = {\textstyle\bigoplus\limits_{m\ge 0}} \; \mu_m \circ \Phi^{(m)}:\; 
\Omega^\dot_{\gr(A)} \;\longrightarrow\; \gr_\cG(\Omega^\dot_A)
$$ 

Now we show the other claimed properties of~$\Phi$. The surjectivity is obvious, as the elements
$dx_{i_1} \wedge \cdots \wedge dx_{i_m}$ generate $\Omega^m_A$. Let us prove that~$\Phi$ is
a homomorphism of anti-commutative algebras. 

By the definition of the filtration~$\cG$, the wedge product in~$\Omega^\dot_A$ is compatible
with~$\cG$, i.e., we have $\cG_i \wedge \cG_j \subseteq \cG_{i+j}$ for all $i,j\ge 0$.
This implies that, for $\omega \in \Omega^k_{\gr(A)}$ and $\eta\in \Omega^m_{\gr(A)}$,
we have
$$
\Phi^{(k+m)}(\omega \wedge \eta)  \;=\; \mu_{k+m}(\bar{\omega} \wedge \bar{\eta}) \;=\;
\overline{\omega\wedge\eta} \;=\; \Phi^{(k)}(\omega) \wedge \Phi^{(m)}(\eta)
$$
Lastly, we check that~$\Phi$ commutes with differentiation. Since~$\Phi$ is $K$-linear,
it suffices to check this for $\omega = \bar{a}\, d\bar{x}_{i_1} \wedge \cdots \wedge 
d\bar{x}_{i_m} \in \Omega^m_{\gr(A)}$
with $\bar{a} \in \gr(A)$ and $i_1 < \cdots < i_m$. On one hand, we have
$d_{\gr(A)}(\omega) = d\bar{a} \wedge d\bar{x}_{i_1} \wedge \cdots \wedge d\bar{x}_{i_m}$, 
and therefore $\Phi^{(m+1)}(d_{\gr(A)}(\omega)) = \overline{da \wedge dx_{i_1} \wedge \cdots \wedge
dx_{i_m}}$. On the other hand, we have $\Phi^{(m)}(\omega) = \overline{a dx_{i_1} \wedge \cdots \wedge dx_{i_m}}$,
and hence $d_{\gr(\Omega)}(\Phi^{(m)}(\omega)) = \overline{da \wedge dx_{i_1} \wedge \cdots \wedge dx_{i_m}}$.
This proves the claim.
\end{proof}

Under some additional hypothesis, the map~$\Phi$ is actually an isomorphism, as the following
proposition demonstrates.

\begin{proposition}[The Graded Ring of the Differentials of a Local Artinian SSCI]\label{prop:IsomGrOmega} 
Let $A = K[x_1,\dots,x_n] / \langle f_1,\dots,f_n\rangle$ be a local 
Artinian SSCI. For $1\le m\le n-1$, assume that the elements
$$
\{ f_i dx_J \mid i=1,\dots,n,\; |J|=m \} \cup \{ df_i \wedge dx_{J'} \mid i=1, \dots, n,\;
|J'| = m-1 \}
$$
form a standard basis of $U_m = I\,\Omega^m_P + dI \wedge \Omega^{m-1}_P$
with respect to~$\cG^{(m)}$.

Then there exists a canonical isomorphism of graded $\gr_\m(A)$-modules
$$
\Phi:\; \Omega^\dot_{\gr_\m(A)} \;\longrightarrow\; \gr_{\cG} ( \Omega^\dot_A )
$$
\end{proposition}

\begin{proof}
In the following, we let $I=\langle f_1,\dots,f_n\rangle$,
and we denote $\ini_\cM(f_i)$ by~$g_i$.
By~\cite[Prop.~2.1]{VV}, we have $\ini_\cM(I) = \langle g_1, \dots, g_n\rangle$ 
and $\gr(A) \cong \gr(P)/ \langle g_1, \dots, g_n\rangle$.  

In view of the lemma, given $m\ge 1$, we want to prove that
$\Phi_m = \mu_m \circ \Phi^{(m)}:\; \Omega^m_{\gr(A)} \longrightarrow \gr_\cG(\Omega^m_A)$
is an isomorphism of $\gr(A)$-modules. We use the epimorphism $\gr(P) \longrightarrow
\gr(P)/\ini_\cG(I) \cong \gr(A)$ and consider the following $\gr(P)$-module presentations.

The $P$-module $\Omega^m_A$ satisfies $\Omega^m_A \cong  \Omega^m_P /U_m$.
This yields 
$$
\gr_{\cG^{(m)}}(\Omega^m_A) \cong \Omega^m_{\gr(P)} / \ini_{\cG^{(m)}}(U_m)
$$
as $\gr(P)$-modules. On the other hand, we have $\Omega^m_{\gr(A)} \cong \Omega^m_{\gr(P)} /V$,
where we let 
$V = \langle \sum_{i=1}^n g_i \Omega^m_{\gr(P)} + \sum_{j=1}^n dg_j \wedge \Omega^{m-1}_{\gr(P)} \rangle$.

Now we check that $V \subseteq \ini_{\cG^{(m)}}(U_m)$. For $i\in \{1,\dots,n\}$
and a homogeneous element $\omega\in \Omega^m_{\gr(P)}$ we have $g_i\,\omega = \ini_\cM(f_i) \, \omega \in 
\ini_\cG^{(m)}(U_m)$. For $j\in \{1,\dots,n\}$ and a homogeneous element $v\in \Omega^{m-1}_{\gr(P)}$
we have $dg_j \wedge v = d(\ini_\cM(f_j)) \wedge v = \ini_\cN(df_j) \wedge v \in \ini_{\cG^{(m)}}(U_m)$.

Thus the map $\Phi_m$ is identified with the canonical epimorphism
$\Omega^m_{\gr(P)}/ V \longrightarrow \Omega^m_{\gr(P)} / \ini_{\cG^{(m)}}(U_m)$, and it remains
to show the inclusion $\ini_{\cG^{(m)}}(U_m) \subseteq V$. 
This follows from the hypothesis which yields that the generators of~$U_m$  
form a standard basis of~$U_m$ with respect to the filtration~$\cG^{(m)}$.
\end{proof}

At this point we are almost ready to prove the desired theorem.
We need one more definition.

\begin{definition}
Let $M,N$ be modules over a ring~$R$ which are equipped with decreasing filtrations
$(\cF_i M)_{i\ge 0}$ and $(\cG_i N)_{i\ge 0}$, respectively.
\begin{enumerate}
\item[(a)] A homomorphism $\phi:\; M \longrightarrow N$ is called a {\bf filtered 
morphism} if $\phi(\cF_i M ) \subseteq \cG_i N$ for all $i\ge 0$.

\item[(b)] A filtered morphism $\phi:\; M \longrightarrow N$ is called {\bf strict}
if $\phi(\cF_i M) = \phi(M) \cap \cG_i N$ for all $i\ge 0$, i.e., if the filtration induced
by~$\phi$ on $\phi(M)$ is the restriction of the filtration of~$N$.
\end{enumerate}
\end{definition}

As shown in~\cite[Ch.~III, \S 2.4]{Bou}, an exact sequence of filtered modules induces
an exact sequence of their graded modules if the maps are strict filtered morphisms.
For finite filtrations, the converse is true, as well, and this allows us to
finally prove the following result.

\begin{theorem}[De Rham Cohomology of Local Artinian SSCI]\label{thm:localSSCI}\, \\
Let $K$ be a field of characteristic 0, let $P=K[x_1,\dots,x_n]$, let 
$I = \langle f_1,\dots,f_n\rangle$ be generated by a super-regular sequence,
and let $A=P/I$ be local ring. Moreover, for $1\le m\le n-1$, assume that  
the elements
$$
\{ f_i dx_J \mid i=1,\dots,n,\; |J|=m \} \cup \{ df_i \wedge dx_{J'} \mid i=1,\dots,n,\;
|J'| = m-1 \}
$$
form a standard basis of $U_m = I\,\Omega^m_P + dI \wedge \Omega^{m-1}_P$.

Then we have $\HdR^0(A)= K$ and $\HdR^m(A) = \{0\}$ for all $m\ge 1$.
\end{theorem}

\begin{proof}
Consider the K\"ahler differential algebra $\Omega^\dot_A$ as a complex
\begin{equation}
0 \;\longrightarrow\; K \;\longrightarrow\; A \;\longrightarrow\; \Omega^1_A 
\;\longrightarrow\; \Omega^2_A \;\longrightarrow \cdots \longrightarrow\; 
\Omega^n_A \;\longrightarrow\; 0 \tag{$\ast$} 
\end{equation}
where we have used $K \subseteq \Ker(d_A)$ to extend the complex to the left.
The maps in this complex are filtered morphisms with respect to~$\cG$, since
for $f\in\cF_{i-m}$ and $dx_J = dx_{j_1}\wedge \cdots \wedge dx_{j_m}$ we have
$d(f\, dx_J) = df \wedge dx_J = \sum_{k=1}^n \frac{\partial f}{\partial x_k}\,
dx_k \wedge dx_J \in \cG_i^{(m+1)}$ because of $\frac{\partial f}{\partial x_k} \in \cF_{i-m-1}$.
In fact, as $\ord_\cF(f) = i-m$ and $f\ne 0$ imply that at least one of the partial derivatives
$\frac{\partial f}{\partial x_k}$ in non-zero, the maps in the above complex are strictly filtered
morphisms.

Therefore we obtain a complex of graded $A$-modules 
\begin{equation}
0 \;\longrightarrow\; K \;\longrightarrow\; \gr_\m(A) \;\longrightarrow\;
\gr_\cG(\Omega^1_A) \;\longrightarrow\cdots \longrightarrow\; \gr_\cG(\Omega^n_A) \;\longrightarrow\; 0
\tag{$\ast\ast$}
\end{equation}
which is isomorphic to the extended complex of $\Omega^\dot_{\gr(A)}$ by Prop.~\ref{prop:IsomGrOmega}.
The latter complex is exact by Prop.~\ref{prop:dRgraded}. Thus the complex $(\ast)$ consisting of modules
with finite filtrations and of strictly filtered morphisms has a graded complex $(\ast\ast)$ which is exact.
It follows that the complex $(\ast)$ is exact, as well (cf.~\cite[Lemma 05QH]{Sta}).
This proves the claim of the theorem.
\end{proof}

Our next example shows that the preceding theorem applies in non-trivial cases.

\begin{example} 
Let $P = \QQ[x,y]$, and consider the ideal $I = \langle f_1,f_2\rangle$, 
where $f_1= x^4$ and $f_2 = y^4 + x^2y^2 + x^3y^2$.
Then $A=P/I$ is clearly a local Artinian super-strict complete intersection.

It is not difficult to verify that $\{ f_1\, dx,\, f_1\, dy,\, f_2\,dx,\, f_2\, dy,\, df_1,\, df_2\}$
is a standard basis of $U= I\Omega^1_P + dI$ with respect to the filtration~$\cG$.
Therefore the theorem applies and shows that $\HdR^0(A) = \Ker(d_A) = \QQ$ and $\HdR^i(A)=\{0\}$ for $i\ge 1$.
Notice that the ideal~$I$ is not quasi-homogeneous, but the ring~$A$ has a presentation $A=P/\tilde{I}$ with a
quasi-homogeneous ideal~$\tilde{I}$.
\end{example}

Finally, we provide an example in which the de Rham cohomology of
a local Artinian super-strict complete intersection~$A$ is trivial,
but in which the canonical epimorphism 
$\Phi:\; \Omega^\dot_{\gr(A)} \longrightarrow \gr_{\cG}(\Omega^\dot_A)$
is not an isomorphism.

\begin{example}
Let $P=\QQ[x,y]$, and consider the ideal $I = \langle f_1,f_2\rangle$, 
where $f_1= x^4$ and $f_2= y^4+x^2y^3+x^3y^2$.
Then $A = P/I$ is a local Artinian ring and $\{f_1,f_2\}$ is a super-regular sequence.

By direct calculation we can verify that $\HdR^0(A) = \QQ$ and $\HdR^i(A)=\{0\}$ for $i\ge 1$.
Here one can prove that the algebra~$A$ has no presentation $A \cong P/\tilde{I}$ 
with a quasi-homogeneous ideal~$\tilde{I}$. 
(For more on this, see the final Example~\ref{ex:NonQuasiHomog}.)

Now it is straightforward to verify $\dim_\QQ(\Omega^1_{\gr(A)}) = 24$, 
$\dim_\QQ({\gr_\cG}(\Omega^1_A)=23$, and that the kernel of $\Phi^{(1)}$ is generated
by the residue class of $3x^2y^3\,dx - 2x^3 y^2\,dy$.
\end{example}

This example indicates that the kernel of the universal derivation may be non-trivial
for various reasons and surprise us even in seemingly clear-cut situations.
Therefore we look at it more closely in the next subsection.

\bigskip
\subsection{The Kernel of the Universal Derivation.}
For $\HdR^0(A) = \Ker(d_A)$, it is clear that $K\subseteq \Ker(d_A)$. However, 
this cohomology group is not always just the base field, as the
following example shows. (It is derived from~\cite[Ex.~1.7]{Ale} 
and its results will be reused in Example~\ref{ex:alex}.)

\begin{example}\label{ex:dR0non-trivial}
Let $P = \QQ[x,y]$, let $I = \langle f_1,\dots,f_5\rangle \subseteq P$ with
\begin{align*}
	f_1 &\;= x y^3  +2 x^3,\quad        && f_2 = x^2 y^2  +\tfrac{5}{3} y^4,\\
	f_3 &\;= y^5 -\tfrac{6}{5} x^4,\quad  && f_4 = x^5 +\tfrac{50}{9} x^3 y, \qquad\qquad   f_5 = y^6,
\end{align*}
and consider $A=P/I$. Using the term ordering {\tt DegRevLex} and Macaulay's Basis Theorem (cf.~\cite[Thm.~1.5.7]{KR2000}),
we calculate that the residue classes of the terms in
$$
\{\, 1,\, x,\, y,\, x^2,\, x y,\, y^2,\, x^3,\, x^2 y,\, x y^2,\, y^3,\, x^4,\, y^4\, \}.
$$
form a $\QQ$-basis of~$A$. In particular, the ring~$A$ is an Artinian local $\QQ$-algebra 
with $\dim_\QQ(A)=12$ and maximal ideal $\langle x,\, y\rangle$. Notice that $x^4\notin I$. 

The computation of a presentation of~$\Omega^1_A$ yields
$$
\Omega^1_A \;=\;  \Omega^1_P / (dI + I\Omega^1_P) \;=\; (P dx\oplus P dy)/ \langle
df_i, f_i dx, f_i dy \mid i=1,\dots,5 \rangle
$$  
where 
\begin{align*}
df_1 &\;=\; (y^3 + 6x^2)\, dx + 3x y^2\, dy, \\ 
df_2 &\;=\; 2x y^2\, dx  +  (2x^2 y + \tfrac{20}{3}\, y^3)\, dy,\\ 
df_3 &\;=\; -\tfrac{24}{5}\, x^3\, dx  +  5y^4\, dy, \\ 
df_4 &\;=\; (5x^4 + \tfrac{50}{3}\, x^2 y)\, dx  + \tfrac{50}{9}\, x^3\, dy, \\
df_5 &\;=\; 6 y^5\, dy.
\end{align*}

Since $x^3\, dx = \tfrac{5}{4}\, f_1\, dx + \tfrac{15}{4}\, f_2\, dy 
-\tfrac{5}{4}\, x\, df_1 - \tfrac{5}{4}\, df_3 \in dI$, it follows that
$x^4\in \Ker(d_A)$. Thus we have $\HdR^0(A)=\Ker(d_A) \subseteq  \QQ \oplus \QQ x^4$,
and using {\tt ApCoCoA} (cf.~\cite{ApCoCoA}), we may check that this is indeed an equality.

A direct calculation shows $\HdR^1(A)=0$, and Corollary~\ref{cor:highestDR}
yields $\HdR^2(A)=0$.
\end{example}

Let us study the question when $K \subseteq \Ker(d_A)$ is a proper inclusion
in greater detail. A first conjecture could be that, if the kernel of~$d_A$
is non-trivial (i.e., it properly contains~$K$) and if~$A$ is a super-strict complete
intersection, then the socle element should be in the kernel of~$d_A$. However, 
the following example shows that this is not the case.

\begin{example}\label{ex:SocEscapeKer}
Let $P=\QQ[x,y]$.
\begin{enumerate}
\item[(a)] Consider the ring $A = P / \langle f_1,f_2\rangle$, where $f_1 = x^5$ and 
$f_2 = y^5 + xy^4 + x^3y^3$. Then~$A$ is a super-strict complete intersection and the socle
of~$A$ is $\soc(A) = \QQ\, x^4y^4$. When we calculate $\HdR^0(A)=\Ker(d_A)$, we
get $\Ker(d_A) = \QQ \oplus \QQ\, x^4 y^4$. So far, so good.

\item[(b)] However, now we let $B = P[z] / \langle f_1, f_2, f_3\rangle$,
where $f_1, f_2$ are as in~(a) and $f_3 = z^2$. Then~$B$ is still a super-strict complete intersection
and $\soc(B) = \QQ\, x^4 y^4 z$. The computation of $\HdR^0(B) = \Ker(d_B)$ yields
$\Ker(d_B)= \QQ \oplus \QQ\, x^4 y^4$, and this kernel intersects the socle of~$B$ trivially.
\end{enumerate}
\end{example}

To get a better understanding of~$\Ker(d_A)$, we let $\cK^\dot$ be the kernel
of the $\gr(A)$-DG-algebra homomorphism~$\Phi$ constructed in Lemma~\ref{lem:Phi}.
In other words, we have an exact sequence of complexes
$$
0 \;\longrightarrow\; \cK^\dot \;\longrightarrow\;  \Omega^\dot_{\gr(A)} 
\;\longrightarrow\; \gr_\cG (\Omega^\dot_A) \;\longrightarrow\; 0 
$$
This sequence induces a long exact cohomology sequence
$$
0 \longrightarrow H^0(\cK^\dot) \longrightarrow H^0(\Omega^\dot_{\gr(A)}) \longrightarrow
H^0(\gr_\cG(\Omega^\dot_A)) \longrightarrow H^1(\cK^\dot) \longrightarrow \cdots
$$
Since the map $\Phi^0:\; \gr(A) \longrightarrow \gr(A)$ is an isomorphism, we have $\cK^0=0$, 
and hence $H^0(\cK^\dot)=0$.  The fact that $\gr(A)$ is a graded $K$-algebra implies
$H^0(\Omega^\dot_{\gr(A)}) =K$ and $H^m(\Omega^\dot_{\gr(A)}) = 0$ for $m\ge 1$.
As $H^0(\gr_\cG(\Omega^\dot_A)) = \Ker(\gr(d_A))$, we obtain a short exact sequence
$$
0 \;\longrightarrow\; K \;\longrightarrow\; \Ker(\gr(d_A)) \;\longrightarrow\; H^1(\cK^\dot)
\;\longrightarrow\; 0
$$
and isomorphisms $H^m(\gr_\cG(\Omega^\dot_A)) \cong H^{m+1}(\cK^\dot)$ for $m\ge 1$.
Let us connect these observations to $\HdR^0(A) = \Ker(d_A)$ as follows.

\begin{theorem}[The Dimension of the 0-th De Rham Cohomology]\label{thm:ZerothDeRham}$\mathstrut$\\ 
Let $A=P/I$ be a local Artinian $K$-algebra as above. Then we have
\begin{align*}
\dim_K \HdR^0(A) &\;=\; \dim_K (\Ker(d_A)) \;=\; \dim_K (\gr_\m(\Ker(d_A)))\\ 
&\;\le\;  \dim_K (\Ker(\gr(d_A))) \;=\; 1 +  \dim_K H^1(\cK^\dot) 
\end{align*}
\end{theorem}

\begin{proof}
To simplify the notation, we drop the names of the filtrations from the operators $\gr(\dots)$
if they are clear from the context.
The first equality follows from the definition of $\HdR^0(A)$, and
the second equality is a consequence of the fact that $\Ker(d_A)$ is a finite dimensional
filtered vector space. As the last equality follows from the above exact sequence, 
it remains to show the inequality.

For the finite dimensional $K$-vector space $\Ker(d_A)$, we have 
$\dim_K(\Ker(d_A)) = \dim_K(\gr(\Ker(d_A)))$ by standard linear algebra.
Given an element $\bar{c} \in \gr(\Ker(d_A))_p$
for some $p\ge 0$, represent it by an element $c\in\Ker(d_A)\cap \m^p$.
The canonical map $\iota_p:\; \gr(\Ker(d_A))_p \longrightarrow 
\gr(A)_p$ satisfies $\iota_p(\bar{c}) = c + \m^{p+1}$.
Then $\gr(d_A)(\iota_p(\bar{c})) = d_A(c) + \cG^{(1)}_{p+1}(\Omega^1_A) = 0$.
Hence the $K$-linear maps $\iota_p$ induce a homogeneous $K$-linear map
$\phi:\; \gr(\Ker(d_A)) \longrightarrow \Ker(\gr(d_A))$.
The desired inequality now follows from the fact that~$\phi$ is injective.
To prove this, let $\bar{c} \in \gr(\Ker(d_A))_p$ such that $\iota_p(\bar{c})=c+\m^{p+1}=0$.
Then we obtain $c\in \Ker(d_A) \cap \m^{p+1}$, and hence $\bar{c}=0$. This
concludes the proof.
\end{proof}

In spite of extensive efforts, we were not able to find an example in which the
inequality in this theorem is strict. On the other hand, we were not able to find
a proof of equality either. Therefore we leave it as an open question for future research
to decide whether the inequality in this theorem can be strict.
An immediate consequence of the preceding theorems is the following hierarchy of
conditions which imply that the kernel of the universal derivation of~$A$ is trivial.

\begin{corollary} \label{cor:H0conditions}
In the setting of the theorem, consider the following conditions.
\begin{enumerate}
\item[(1)] The ideal~$I$ is generated by a strict regular sequence $\{f_1,\dots,f_n\}$
and $\{df_i, f_i\, dx_j \mid i,j=1,\dots,n\}$ is a standard basis of~$dI + I\,\Omega^1_P$.

\item[(2)] $\Ker(\gr(d_A)) = K$

\item[(3)] $H^1(\cK^\dot) = 0$

\item[(4)] $\Ker(d_A)=K$
\end{enumerate}
Then we have (1) $\Rightarrow$ (2) $\Leftrightarrow$ (3) $\Rightarrow$ (4).
\end{corollary}

\begin{proof}
Here (1)$\Rightarrow$(2) follows from Prop.~\ref{prop:IsomGrOmega},
as $\Ker(\gr(d_A)) = \Ker( \Phi^{(1)} \circ d_{\gr(A)} ) = \Ker(d_{\gr(A)}) = K$.
The equivalence (2)$\Leftrightarrow$(3) is a consequence of the short exact
sequence immediately preceding Theorem~\ref{thm:ZerothDeRham}, and (3)$\Rightarrow$(4)
follows from that theorem.
\end{proof}

\bigbreak
%
%

\section{The De Rham Cohomology of a 0-Dimensional Scheme}\label{sec4}

Building upon the foundation laid in the preceding sections, we now turn to
finding the de Rham cohomology of the coordinate ring of a 0-dimensional scheme. 
Throughout this section, we work over a base filed~$K$ of characteristic zero.
The subtle and more complicated phenomena in characteristic~$p$ will be left to future
explorations. Let us start with the projective case, because it is an immediate consequence
of Prop.~\ref{prop:dRgraded}.

\medskip
\subsection{Projective 0-Dimensional Schemes.} Let $K$ be a field of characteristic zero,
let $P=K[x_0, x_1, \dots, x_n]$ be standard graded, and let~$I$ be a 0-dimensional 
saturated homogeneous ideal in~$P$. Then~$I$ is the homogeneous vanishing ideal of the 
0-dimensional subscheme $\X=\mathrm{Proj}(R)$ of~$\mathbb{P}^n_K$. The homogeneous coordinate ring
$R=P/I$ of~$\X$ is a standard graded 1-dimensional Cohen-Macaualy ring.
Therefore Prop.~\ref{prop:dRgraded}, immediately yields the following corollary.

\begin{corollary}[De Rham Cohomology of the Homogeneous Coordinate Ring]\label{cor:DeRhamCohProj}\, \\
For the homogeneous coordinate ring $R=P/I$ of a 0-dimensional scheme in~$\mathbb{P}^n_K$, we have
$\HdR^0(R) = K$ and $\HdR^m(R) = 0$ for $m\ge 1$.
\end{corollary}

Let us denote the set of closed points of~$\X$ by $\Supp(\X)=\{p_1,\dots,p_s\}$.
Since~$K$ is infinite, we may assume w.l.o.g.\ that no point in the support of~$\X$ lies on 
the hyperplane $\mathcal{Z}(x_0)$. 
Under this assumption, the elements $x_0$ and $x_0-1$ are non-zerodivisors of~$R$.

The residue class ring $\overline{R} = R / \langle x_0\rangle $ is a 0-dimensional 
graded $K$-algebra, and thus a finite-dimensional $K$-vector space.
It is called the \textbf{Artinian reduction} of~$R$. Again, Prop.~\ref{prop:dRgraded}
shows that the de Rham cohomology of~$\overline{R}$ is tirival.

\begin{corollary}[De Rham Cohomology of the Artinian Reduction]\label{cor:DRofArtRed}\, \\
For the Artinian reduction $\overline{R}$ of the homogeneous coordinate ring of a 0-dimensional		
scheme in~$\mathbb{P}^n_K$, we have $\HdR^0(\overline{R})=K$ and $\HdR^m(\overline{R})=0$ for $m\ge 1$.
\end{corollary}

The residue class ring $S=R/\langle x_0-1\rangle$ is the coordinate ring of~$\X$
in the affine space $\mathbb{A}^n_K \cong D_+(x_0)$.
The canonical $K$-algebra epimorphism $\psi: R \longrightarrow S$ induces a homomorphism 
of graded differential algebras $\Psi:  \Omega^\dot_R \longrightarrow \Omega^\dot_S$.
By \cite[Rules 2.8(e)]{Kun1986}, the map~$\Psi$ induces a homomorhism of graded $K$-algebras
$\HdR(\Psi): \HdR(R)\longrightarrow \HdR(S)$. However, as we shall see in the next subsection,
this map is not an isomorphism in general.

\medskip
\subsection{Affine 0-Dimensional Schemes.}

Now let $K$ be a field of characteristic zero, let $P=K[x_1,\dots,x_n]$,
let~$I$ be a 0-dimensional ideal in~$P$, and let $S=P/I$.
Then $\X = \mathcal{Z}(I)$ is a 0-dimensional subscheme of $\mathbb{A}^n_K$
and~$S$ is the affine coordinate ring of~$\X$.

\begin{remark}(Reduction to the Local Case)\label{rem:ReduceToLocal}\, \\
In the ring~$S$, we let $\langle 0\rangle = \q_1 \cap \cdots \cap \q_s$ be the primary decomposition
of the zero ideal. For $i=1,\dots,s$, let $\m_i=\Rad(\q_i)$, so that $\m_1,\dots,\m_s$
are the maximal ideals of~$S$.

Then the Chinese Remainder Theorem yields an isomorphism 
$S \cong A_1 \times \cdots \times A_s$, where $A_i = S/\q_i = \mathcal{O}_{\X,p_i}$ 
is the local ring of~$\X$ at the corresponding point~$p_i$ of its support. 
The rings $A_1, \dots, A_s$ are Artinian local rings. By Prop.~\ref{prop:functorial}.b,
we have an isomorphism of graded $K$-algebras
$$
\HdR(S) \;\cong\; \HdR(A_1) \times \cdots \times \HdR(A_s)
$$
In particular, $\HdR^0(S)\cong \prod_{i=1}^s \HdR^0(A_i)$ shows $\dim_K(\HdR^0(S)\ge s$. 
\end{remark}

In view of this remark, we now assume that $\X=\mathcal{Z}(I)$ is a 0-dimensional scheme 
in~$\mathbb{A}^n_K$ concentrated at one point, and that its affine coordinate ring is $A=P/I$, 
where~$I$ is an $\M$-primary ideal for some maximal ideal~$\M$ of~$P$.

\begin{proposition}[Reduction to $K$-Rational Support]\label{prop:ReduceToKRat}\, \\
Let $A=P/I$ be an Artinian local $K$-algebra with maximal ideal $\m=\M+I$, where~$\M$ is a maximal
ideal of~$P$ and~$I$ is $\M$-primary.
\begin{enumerate}
\item[(a)] There exists a Galois extension~$L$ of~$K$ such that
$\M \otimes_K L = \N_1 \cap \cdots \cap \N_r$ with linear maximal ideals $\N_1,\dots,\N_r$
in $\overline{P} = L[x_1,\dots,x_n]$.

\item[(b)] We have $I \otimes_K L = \Q_1 \cap \cdots \cap \Q_r$ with $\N_i$-primary 
ideals~$\Q_i$ in~$\overline{P}$, and we obtain an isomorphism
$A \otimes_K L \cong B_1 \times \cdots \times B_r$
with $B_i = \overline{P} / \Q_i$ for $i=1,\dots,r$.

\item[(c)] There exist an $\langle x_1,\dots, x_n\rangle$-primary
ideal~$\Q$ in~$\overline{P}$ and $L$-algebra isomorphisms $B\cong B_1 \cong \cdots \cong B_r$
for $B = \overline{P}/\Q$.

\item[(d)] The de Rham cohomology of~$A\otimes_K L$ satisfies 
$$
\HdR(A)\otimes_K L  \;\cong\; \HdR(B) \times \cdots \times \HdR(B)
$$
where~$L$ and~$B$ are chosen as above, and where there are~$r$ factors on the right-hand side. 

\end{enumerate}
\end{proposition}

\begin{proof} To prove~(a), we note that $P/\M$ is a finite field extension of~$K$. 
Hence the Primitive Element Theorem yields $P/\M \cong K[x]/\langle f\rangle$ 
for some irreducible polynomial $f\in K[x]$. 
As~$f$ is separable, we can choose the splitting field~$L \supseteq K$ of~$f$
and get that~$f$ is a product of distinct linear factors $\ell_1,\dots,\ell_r \in L[x]$.
This yields 
$$
(P/\M) \otimes_K L \;\cong\; 
L[x]/\langle \ell_1\rangle \times \cdots\times L[x]/ \langle \ell_r\rangle
$$
For $i=1,\dots,r$, the kernel $\N_i$ of the $L$-algebra epimorphism $\epsilon_i:\; \overline{P} \longrightarrow
L[x]/ \langle \ell_i \rangle$ is a linear maximal ideal of~$\overline{P}$. Then the kernel $\M \otimes_K L$
of the $L$-algebra epimorphism $\eta:\; \overline{P}  \longrightarrow \prod_{i=1}^r L[x] / \langle \ell_i\rangle$
equals $\N_1 \cap \cdots \cap \N_r$, as claimed.

Next we show~(b). Since the ideal~$I$ is $\M$-primary, it follows from~(a) that the prime components of the ideal
$I\otimes_K L$ in~$\overline{P}$ are among $\N_1,\dots,\N_r$. As $L/K$ is a Galois extension, all prime
components appear and the primary decomposition of $I\otimes_K L$ has the desired form
$I\otimes_K L = \Q_1 \cap \cdots \cap \Q_r$. Hence the Chinese Remainder Theorem yields the
claimed isomorphism $A \otimes_K L \cong \prod_{i=1}^r \overline{P} / \Q_i$.

For the proof of~(c), we note that the Galois group of~$L/K$ permutes the factors of~$A$ transitively.
This yields isomorphisms of $L$-algebras $B_i \cong B_j$ for $i,j=1,\dots,r$. Since the ideals $\N_i$
are linear maximal ideals, we find linear changes of coordinates identifying them with 
$\langle x_1,\dots,x_n\rangle$ in~$\overline{P}$. Let $\Q$ be the image of~$\Q_1$ under the
isomorphism $\N_1\cong \langle x_1,\dots,x_n\rangle$. Then the isomorphism $\overline{P}/\Q_1 \cong
\overline{P}/\Q$ induces isomorphisms $\overline{P}/\Q_i \cong \overline{P}/\Q$ for $i=2,\dots,r$,
as was to be shown.

Finally, to show~(d), we note that the isomorphisms of~(b) and~(c), together with
Prop.~\ref{prop:functorial}.a,b, yield the claim. 
\end{proof}

This proposition and the remark preceding it allow us to reduce the computation of
the de Rham cohomology of an affine 0-dimensional scheme to the cases studied in
Section~\ref{sec3} after a suitable base field extension and a linear change of coordinates. 
Let us spell out some consequences explicitly. We shall say that~$\X = \mathcal{Z}(I)$ is a \textbf{fat point scheme}
if there are maximal ideals $\M_1,\dots,\M_s$ in $P=K[x_1,\dots,x_n]$ and positive integers
$\nu_1,\dots,\nu_s \in \NN_+$ such that $I = \M_1^{\nu_1} \cap \cdots \cap \M_s^{\nu_s}$.

\begin{proposition}[De Rham Cohomology of Fat Point Schemes]\label{prop:HdRofFatPoints}$\mathstrut$\\
Let $K$ be a field of characteristic zero, let $\M_1,\dots,\M_s$ be maximal ideals
in $P=K[x_1,\dots,x_n]$, let $\nu_1,\dots,\nu_s \in \NN_+$, let $\X = \mathcal{Z}(I)$
be the fat point scheme in $\AA^n_K$ defined by $I = \M_1^{\nu_1} \cap \cdots \cap \M_s^{\nu_s}$,
and let $S=P/I$ be the affine coordinate ring of~$\X$. Moreover, let $\kappa_i=P/\M_i$
for $i=1,\dots,s$.

Then we have $\HdR^0(S) \cong \kappa_1 \times\cdots \times \kappa_s$ and $\HdR^m(S)=0$ for $m\ge 1$.
\end{proposition}

\begin{proof}
Using Remark~\ref{rem:ReduceToLocal}, we see that it suffices to consider the case $s=1$,
i.e., the case $A=P/\M^\nu$. Since this is a finitely generated algebra over the field
$\kappa=P/\M$, we know that $\kappa \subseteq \Ker(d_A)$.

Now we apply Prop.~\ref{prop:ReduceToKRat}. We find a Galois extension $L\supseteq K$
such that $\M \otimes_K L = \N_1 \cap \cdots \cap \N_r$ with linear maximal ideals $\N_i$.
Moreover, the rings $B_i= (P \otimes_K L) / \N_i^{\nu}$ are all isomorphic to
$B = (P \otimes_K L) / \N^\nu$ with $\N= \langle x_1,\dots,x_n \rangle$ in $P\otimes_K L = L[x_1,\dots,x_n]$.
Using Prop.~\ref{prop:dRFatPoint}, we obtain
$$
\HdR^0(A \otimes_K L) \cong L^r \hbox{\qquad \rm and\qquad} \HdR^m(A \otimes_K L)=0
\hbox{\qquad\rm for\quad} m\ge 1.
$$
Hence we get $\dim_K(\HdR^0(A)) = \dim_L(\HdR^0(A \otimes_K L))= r = \dim_L((P\otimes_K L)/(\M\otimes_K L))
= \dim_K(P/\M)=\dim_K(\kappa)$, and this shows that the inclusion $\kappa \subseteq \Ker(d_A)$
is an equality.
\end{proof}

It seems worth pointing out that the preceding proposition includes the case
of affine reduced 0-dimensional schemes, i.e., of finite sets of points in~$\AA^n_K$.

\begin{corollary}[De Rham Cohomology of Finite Sets of Points]\label{cor:dRofPointSets}$\mathstrut$\\
Let $\X$ be a reduced 0-dimensional subscheme of~$\AA^n_K$ and $\Supp(\X) = \{ p_1,\dots,p_s\}$.
For $i=1,\dots,s$, let $\kappa_i$ be the residue field of the local ring $\mathcal{O}_{\X,p_i}$,
and let~$S$ be the affine coordinate ring of~$\X$.

Then we have $\HdR^0(S) \cong \kappa_1 \times \cdots \times \kappa_s$ and $\HdR^m(S)=0$ for $m\ge 1$.
\end{corollary}

As we saw already in Example~\ref{ex:dR0non-trivial}, for arbitrary
0-dimensional schemes, we may have $\HdR^0(S)\ne K$, even if $\charac(K)=0$.
Example~\ref{ex:dR0non-trivial} is a simplification of~\cite[Ex.~1.7]{Ale} which is a case where
the support of~$\X$ consists of several points. Let us analyze the more complicated
example with the methods  developed here.

\begin{example}\label{ex:alex}
Let $P=\QQ[x_1,x_2]$, let $f=x_1^4 +x_1^2 x_2^3 +x_2^5 \in P$, and let
$I=\langle \frac{\partial f}{\partial x_1}, \frac{\partial f}{\partial x_2} \rangle$.
Then the 0-dimensional scheme $\X=\mathcal{Z}(I)$ has the affine coordinate 
ring $S=P/I$ and length $\dim_\QQ(S)=14$.

The primary decomposition of the zero ideal in~$S$ is $\langle 0\rangle =  \q_1
\cap \q_2$ where $\q_1$ is the residue class ideal of the ideal $\langle f_1,\dots,f_5\rangle$
given in Example~\ref{ex:dR0non-trivial}, and where $\q_2$
is the residue class ideal of $\langle 3x_2-10,\, \tfrac{100}{9}x_1^2 +\tfrac{50000}{243} \rangle$.
We have already seen that $\Spec(S/\q_1)$ is a non-reduced point of length~12.
The ideal~$\q_2$ defines a reduced point whose residue class field satisfies $\dim_\QQ(\kappa_2)=2$.

Now we use Remark~\ref{rem:ReduceToLocal} and get $\HdR^0(S)\cong \HdR^0(S/\q_1) \times
\HdR^0(S/\q_2)$. For the first factor, we saw that $\dim_\QQ(\HdR^0(S/\q_1))=2$. For the second
factor, Corollary~\ref{cor:dRofPointSets} yields $\dim_\QQ(\HdR^0(S/\q_2)) = 2$.
Altogether, we obtain $\dim_\QQ(\HdR^0(S))=4$.

Notice that A.G.~Alexandrov found slightly different numbers in~\cite{Ale}, since he was working
over the convergent power series ring $\mathbb{C}\{\{ x_1,x_2\}\}$.
\end{example}

Finally, we point out that the de Rham cohomology of the affine coordinate ring of~$\X$
may be trivial even if the local rings of~$\X$ are not quasi-homogeneous.
Our concluding example is a case in point.

\begin{example}\label{ex:NonQuasiHomog}
In $P=\QQ[x,y]$, let $I = \langle f_1,\, f_2\rangle$ with $f_1 = x^5$ and $f_2 = y^5+x^4y^2 + x^2 y^4$.
The scheme $\X = \mathcal{Z}(I)$ is 0-dimensional and $\Supp(\X) = \{(0,0)\}$. 

First we determine the de Rham cohomology of~$S=P/I$. It is clear that $\{f_1, f_2\}$ is a super
regular sequence, and thus $\dim_\QQ(S)=25$. Moreover, using ApCoCoA (cf.~\cite{ApCoCoA}) or direct calculation,
we find $\dim_\QQ(\Omega^1_S)=38$ and $\dim_\QQ(\Omega^2_S)=14$. By Corollary~\ref{cor:highestDR}, this implies
$\dim_\QQ(\Im(\delta_1))=14$. Next we compute
\begin{align*}
\HdR^0(S) \;=\;& \Ker(d_S) \;=\; \{ h\in P \mid d_S(\bar{h}) = 0 \}\\ 
\;=\;&  \{ h\in P \mid \tfrac{\partial h}{\partial x}\, dx + \tfrac{\partial h}{\partial y}\, dy
\in P df + P dg + I\, \Omega^1_P \}
\end{align*}
and get $\HdR^0(S)=\QQ$. This yields $\dim_\QQ(\Im(d_S))= \dim_\QQ(S) - 1 =24$,
and since $\dim_\QQ(\Ker(\delta_1)) = \dim_\QQ(\Omega^1_S) - \dim_\QQ(\Im(\delta_1)) = 
38 - 14 = 24$, it follows that $\dim_\QQ(\HdR^1(S)) = \dim_\QQ(\Ker(\delta_1)) -   
\dim_\QQ(\Im(d_S)) = 0$. Thus the ring~$S$ has trivial de Rham cohomology.

Now we check that~$S$ is not quasi-homogeneous. We have to prove that there is no
$\QQ$-algebra isomorphism $\phi:\; P \longrightarrow P$ such that $\phi(I)$ is
homogeneous with respect to a grading given by positive weights $(w_x,w_y) \in \NN_+^2$.
To make $\phi(f_2)$ homogeneous, we have to eliminate the term~$x^4 y^2$. Hence we define
a $\QQ$-algebra automorphism $\psi:\; P \longrightarrow P$ by $\psi(x) = \tilde{x}$
and $\phi(y) = \tilde{y} - \frac{1}{5}\, \tilde{x}^2$. This yields
$$
\psi(f_2) \;=\; \tilde{y}^5 + \tilde{x}^4\, \tilde{y}^2 + \tfrac{2}{5}\, \tilde{x}^4\, \tilde{y}^3 + 
\hbox{\;\rm higher degree terms}
$$
Clearly, there is no way to assign degrees to~$\tilde{x}$ and~$\tilde{y}$ such that the second and
third term of~$\psi(f_2)$ have the same degree. Consequently, the ring~$S$ is not quasi-homogeneous.
\end{example}

\bigskip
\subsection*{Acknowledgments.} The first author thanks Hue University of Education (Vietnam)
for its hospitality during part of the preparation of this paper. 
This work was supported by the \textit{Vietnam Ministry of Education and Training}
under the grant number B2026-DHH-01.

\bigbreak
%
%


\begin{thebibliography}{10}

\bibitem{Ale}
A.G.\ Aleksandrov, \textit{Analytic invariants of multiple points},
Methods and Applications of Analysis 25 (2018), 167-204.

		
\bibitem{ApCoCoA}
The ApCoCoA Team, Applied Computations in Computer Algebra, available at\\
{\tt https://apcocoa.uni-passau.de}
	

\bibitem{Bou}
N.\ Bourbaki, {\it Commutative Algebra}, Hermann Publ., Paris 1972. 


\bibitem{BK} M.\ Brion and S.\ Kumar, {\it Frobenius Splitting Methods in Geometry and
Representation Theory}, Progress in Math.\ {\bf 231}, Birkh\"auser; Boston, 2005.	

	
\bibitem{BH1993}
W.\ Bruns and J. Herzog, \textit{Cohen-Macaulay Rings},  
Cambridge Stud. Adv. Math., vol. 39, Cambridge University Press, Cambridge, 1993.
	
	
\bibitem{DK1999} G.\ de Dominicis and M.\ Kreuzer, K\"ahler differentials for
points in~$\mathbb{P}^n$, J.\ Pure Appl.\ Alg.\ {\bf 141} (1999), 153-173.
	
	
\bibitem {Har1975}
R.\ Hartshorne, On the de Rham cohomology of algebraic varieties, 
Inst. Hautes \'{E}tudes Sci. Publ. Math. \textbf{45} (1975), 5--99.

	
\bibitem {KLL2019}
M.\ Kreuzer, T.N.K.\ Linh, and L.N.\ Long,
K{\"a}hler differential algebras for $0$-dimensional schemes,
J. Algebra \textbf{501} (2019), 255-284.

	
\bibitem {KLL2021}
M. Kreuzer, T.N.K. Linh and L.N.\ Long,
Hilbert polynomial of K{\"a}hler differential modules
for fat point schemes, Acta Math. Vietnam. \textbf{46} (2021), 441-455.

	
\bibitem {KLL2025}
M. Kreuzer, T.N.K. Linh, and L.N. Long, Differential theory of zero-dimensional 
schemes, J.\ Pure Appl.\ Algebra {\bf 229} (2025), paper 107815
	

\bibitem {KR2000}
M.\ Kreuzer and L.\ Robbiano, \textit{Computational Commutative Algebra 1},
Springer-Verlag, Heidelberg, 2000.

	
\bibitem {KR2005}
M.\ Kreuzer and L.\ Robbiano, \textit{Computational Commutative Algebra 2},
Springer-Verlag, Heidelberg, 2005.

	
\bibitem {KR2016}
M.\ Kreuzer and L.\ Robbiano, \textit{Computational Linear and Commutative Algebra},
Springer Int.\ Publ., Cham, 2016.

	
\bibitem {Kun1986}
E. Kunz, \textit{K\"{a}hler Differentials},
Advanced Lect.\ in Math., Vieweg Verlag, Braunschweig, 1986.
	

\bibitem{Rei} H.-J.\ Reiffen, Das Lemma von Poincar\'e f\"ur holomorphe
Differentialformen auf komplexen R\"aumen, Math.\ Zeitschr.\ {\bf 101} (1967), 269--284.


\bibitem{Sal} J.\ Sally, Super-regular sequences, Pacific.\ J.\ Math.\ {\bf 84} (1979),
465--481.


\bibitem{Sch2012}
P.\ Scheiblechner, Effective de Rham cohomology: the hypersurface case, 
ISSAC 12: Proc.\ 37th Int.\ Symp.\ on Symbolic and Algebraic Computation, 
ACM, New York, 2012, pp.\ 305-310.


\bibitem{Sta}
The {S}tacks project authors, {\it The Stacks Project}, 2026, available at \\
{\tt https://stacks.math.columbia.edu}


\bibitem{VV} 
P.\ Valabrega and G.\ Valla, Form rings and regular sequences,
Nagoya Math.\ J.\ {\bf 72} (1978), 93--101.


\bibitem {Wei}
C.A.\ Weibel, \textit{An Introduction to Homological Algebra}, Cambridge Univ.\ Press, Cambridge, 1994.
	
\end{thebibliography}
\end{document}